\RequirePackage{fix-cm}
\PassOptionsToPackage{numbers,sort&compress,square}{natbib}
\documentclass[smallextended,natbib]{svjour3}
\bibpunct[,]{[}{]}{,}{n}{,}{,}

\usepackage[T1]{fontenc}
\usepackage[english]{babel}
\usepackage[autostyle]{csquotes}
\usepackage{xcolor}
\usepackage{graphicx}
\usepackage{fixltx2e}
\usepackage{bm}
\usepackage{microtype}
\usepackage{mathtools}
\usepackage{amssymb,amsmath,amsfonts}
\usepackage{leftidx}
\usepackage{booktabs}
\usepackage{enumitem}

\usepackage{pgfplots}
\pgfplotsset{compat=newest}
\usetikzlibrary{plotmarks}
\usepackage{grffile}
\pgfplotsset{plot coordinates/math parser=false}
\newlength\figureheight
\newlength\figurewidth
\newlength\dotfigureheight
\newlength\dotfigurewidth
\newlength\dotsize
\definecolor{tol00}{HTML}{AA4499}
\definecolor{tol01}{HTML}{882255}
\definecolor{tol02}{HTML}{CC6677}
\definecolor{tol03}{HTML}{DDCC77}
\definecolor{tol04}{HTML}{999933}
\definecolor{tol05}{HTML}{117733}
\definecolor{tol06}{HTML}{44AA99}
\definecolor{tol07}{HTML}{88CCEE}
\definecolor{tol08}{HTML}{332288}

\colorlet{dotcolor}{tol00}
\colorlet{figcolor0}{tol00}
\colorlet{figcolor1}{tol04}
\colorlet{figcolor2}{tol08}

\usepackage{hyperref}
\usepackage{doi}

\usepackage{lineno}
\newcommand*\patchAmsMathEnvironmentForLineno[1]{%
  \expandafter\let\csname old#1\expandafter\endcsname\csname #1\endcsname
  \expandafter\let\csname oldend#1\expandafter\endcsname\csname end#1\endcsname
  \renewenvironment{#1}%
     {\linenomath\csname old#1\endcsname}%
     {\csname oldend#1\endcsname\endlinenomath}}%
\newcommand*\patchBothAmsMathEnvironmentsForLineno[1]{%
  \patchAmsMathEnvironmentForLineno{#1}%
  \patchAmsMathEnvironmentForLineno{#1*}}%
\AtBeginDocument{%
\patchBothAmsMathEnvironmentsForLineno{equation}%
\patchBothAmsMathEnvironmentsForLineno{align}%
\patchBothAmsMathEnvironmentsForLineno{flalign}%
\patchBothAmsMathEnvironmentsForLineno{alignat}%
\patchBothAmsMathEnvironmentsForLineno{gather}%
\patchBothAmsMathEnvironmentsForLineno{multline}%
}

\providecommand{\versionortoday}{\today}

\newcommand{\bi}{{\bm{\mathrm{i}}}}
\newcommand{\bj}{{\bm{\mathrm{j}}}}
\newcommand{\bk}{{\bm{\mathrm{k}}}}
\newcommand{\Gsp}{G_{\text{sparse}}}
\newcommand{\Gdn}{G_{\text{dense}}}

\newcommand{\dltx}{\prescript{}{\Delta}{x}}
\newcommand{\dltN}{\prescript{}{\Delta}{N}}
\newcommand{\ebvp}{e_{\text{BVP}}}
\newcommand{\einterp}{e_{\text{interp}}}

\newtheorem{assumption}{Assumption}

\AtBeginDocument{\renewcommand\Re{\mathbb{R}}}
\AtBeginDocument{}
\AtBeginDocument{\renewcommand{\vec}[1]{\bm{#1}}}

\DeclareMathOperator*{\argmin}{arg\,min}
\DeclareMathOperator*{\argmax}{arg\,max}

\newcommand{\der} [2]{\frac{d#1}{d#2}}
\newcommand{\pder} [2]{\frac{\partial#1}{\partial#2}}

\DeclarePairedDelimiter{\abs}{\lvert}{\rvert}%
\DeclarePairedDelimiter{\norm}{\lVert}{\rVert}

\makeatletter
\let\oldabs\abs
\def\abs{\@ifstar{\oldabs}{\oldabs*}}
\let\oldnorm\norm
\def\norm{\@ifstar{\oldnorm}{\oldnorm*}}
\makeatother

\providecommand\given{}
\newcommand\SetSymbol[1][]{\nonscript\:#1\vert\nonscript\:\allowbreak}
\DeclarePairedDelimiterX\Set[1]\{\}{%
  \renewcommand\given{\SetSymbol[\delimsize]}
  #1
}

\journalname{myjournal}

\titlerunning{Sparse Grid Characteristics Method for Optimal Control and
  HJB Equations}
\title{Mitigating the Curse of Dimensionality: Sparse Grid Characteristics
  Method for Optimal Feedback Control and HJB Equations
  \thanks{This document has been approved for public release;
    its distribution is unlimited.}
}
\author{Wei Kang \and Lucas~C. Wilcox}
\date{\versionortoday}

\institute{Wei Kang \and Lucas~C. Wilcox \at{} Department of Applied Mathematics,
  Naval Postgraduate School, Monterey, CA \\
  \email{\{wkang,lwilcox\}@nps.edu}
}

\begin{document}

\maketitle

\begin{abstract}
  We address finding the semi-global solutions to optimal feedback control and the Hamilton--Jacobi--Bellman
  (HJB) equation.
  Using the solution of an HJB equation,
  a feedback optimal control law can be implemented in real-time with minimum
  computational load.
  However, except for systems with two or three state variables, using
  traditional techniques for numerically finding a semi-global solution to an
  HJB equation for general nonlinear systems is infeasible due to the curse of
  dimensionality.
  Here we present a new computational method for finding feedback optimal control and solving HJB equations
  which is able to mitigate the curse of dimensionality.
  We do not discretize the HJB equation directly, instead we introduce a sparse
  grid in the state space and use the Pontryagin's maximum principle to derive a set
  of necessary conditions in the form of a boundary value problem, also known as
  the characteristic equations, for each grid point.
  Using this approach, the method is spatially causality free, which enjoys the
  advantage of perfect parallelism on a sparse grid.
  Compared with dense grids, a sparse grid has a significantly reduced size
  which is feasible for systems with relatively high dimensions, such as the
  $6$-D system shown in the examples.
  Once the solution obtained at each grid point, high-order accurate polynomial
  interpolation is used to approximate the feedback control at arbitrary points.
  We prove an upper bound for the approximation error and approximate it
  numerically.
  This sparse grid characteristics method is demonstrated with two examples of
  rigid body attitude control using momentum wheels.
  keywords{optimal feedback control \and Hamilton--Jacobi--Bellman equation \and
  sparse grid \and method of characteristics \and rigid body attitude control}
\end{abstract}

\section{Introduction}
The optimal feedback control of nonlinear systems is a challenging problem.
Using dynamic programming, the feedback optimal control law is constructed based
on the solution of a partial differential equation (PDE) that is called the
Hamilton--Jacobi--Bellman (HJB) equation.
This theoretically elegant approach suffers some difficulties in computation due
to the curse of dimensionality, a term that was coined by Richard~E. Bellman
when considering problems in dynamic optimization, which relates to the fact
that the size of the discretized problem in solving HJB equations increases
exponentially with the dimension.
Finding an approximate solution to HJB-type of equations in a local neighborhood
of a trajectory has been extensively studied, see for example
\citet{albrekht},
\citet{cacace},
\citet{kang-isidori},
\citet{lukes},
\citet{navasca},
\citet{falcone2},
and the references therein.
Some of the previously proposed methods can be applied to systems with high
dimensions.
However, finding semi-global solutions to HJB equations, i.e., solutions
satisfying a required accuracy in a given domain, faces the curse of
dimensionality.

We present a new computational approach to finding semi-global solutions of
optimal feedback control and HJB equations for nonlinear systems. The method is
designed to mitigate the curse of dimensionality.
Our approach does not discretize the HJB equation directly, instead we introduce
a sparse grid in the state space and use Pontryagin's maximum principle (PMP) to
derive a set of necessary conditions in the form of a boundary value problem
(BVP), also known as the characteristic equations, for each grid point
independently.
For details on PMP and its relationship with HJB equations, the reader
is referred to \citet{pontryagin}, \citet{barron-jensen}, and
\citet{bardi-capuzzo}.
The curse of dimensionality is mitigated by using sparse grids. The idea of
sparse grids is based on Smolyak's work~\cite{smolyak}.  The reader is also
referred to \citet{bungartz}, \citet{garcke}, and \citet{zenger} for more
details and different perspectives. Relative to a dense grid, the size
of the sparse grid is significantly reduced. In the propose computational
method, the solution at each grid point is found using a Lobatto~IIIa method
that solves a two-point BVP\@. Different from many existing algorithms of
solving PDEs, this approach is not based on spatial causality. A significant
advantage of this causality free method lies in its perfect parallelism, a
desirable property for modern computational equipment with many-core clusters.
Some results on the rate of convergence and the upper bound of approximation
error are proved.

The method is exemplified by two numerical examples of rigid body attitude
control using momentum wheels. In this case, the HJB equation is six
dimensional. The second example in itself is interesting, in which the system
has two pairs of momentum wheels. The system is uncontrollable, which makes
optimal feedback control difficult.
To the best of our knowledge, no solutions have been found for the HJB equation of this
problem.

\section{Problem Formulation}
\label{sec_formulation}
An optimal feedback control law is a function
\begin{equation*}
  u^{*}(t,x)
\end{equation*}
that minimizes the cost functional
\begin{equation*}\label{sat_cost}
  \int_{t}^{T} L(t,\vec{x}(s),\vec{u}(s))\; ds + h(\vec{x}(T))
\end{equation*}
subjecting to the control system
\begin{subequations}\label{ode}
  \begin{alignat}{2}
    \dot{\vec{x}}(s) &= \vec{f}(s,\vec{x}(s),\vec{u}(s)) && \qquad (t < s < T), \\
    \vec{x}(t) &= x.
  \end{alignat}
\end{subequations}
In this formulation,
$\vec{x} : (t,T)\to\Re^n$ is the state,
$\vec{u} : (t,T)\to\Re^m$ is the control,
$L : (t,T)\times\Re^n\times A\to\Re$ is running cost,
$h : (t,T)\times\Re^n\to\Re$ is the final cost,
$\vec{f} : (t,T)\times\Re^n\times A\to\Re^n$ is a bounded, Lipschitz continuous
function, and
$A\subset\Re^m$ is a compact set.
Here $t\in\Re$ is the initial time and $x\in\Re^n$ is the initial condition.
For the simplicity of discussion, we assume that the final
time, $T\in\Re$, is fixed.  Following the standard approach in optimal
control, we define the Hamiltonian
\begin{equation*}
  H(t,x,\lambda,u)=L(t,x,u)+\lambda^\mathsf{T} f(t,x,u)
\end{equation*}
where $\lambda\in\Re^n$ is the costate and $u\in A$. The function
\begin{equation*}\label{ustar}
  u^\ast(t,x,\lambda) = \argmin_{u} H(t,x,\lambda,u)
\end{equation*}
minimizes the Hamiltonian. The value function is defined as
\begin{equation*}
  V(t, x) = \inf_{\vec{u}} \int_{t}^{T} L(s,\vec{x}(s),\vec{u}(s))\; ds +
  h(\vec{x}(T)) \qquad (0 \le t < T),
\end{equation*}
where $\vec{x}(s)$ satisfies the identity~\eqref{ode}.  Further $V$ satisfies
the HJB equation, a PDE, with a final time condition given as
\begin{subequations}\label{sat_HJB}
\begin{alignat}{2}
  V_t(t,x)+H^\ast\left(t,x,V_x^\mathsf{T}(t,x)\right)&=0 &&\qquad (x\in\Re^n,0 \le t < T),  \\
  V(T,x)&=h(x) &&\qquad (x\in\Re^n),
\end{alignat}
\end{subequations}
where
$H^\ast(t,x,\lambda)=H(t,x,\lambda,u^\ast)$,
$V_t=\frac{\partial V}{\partial t}$, and
$V_x=\frac{\partial V}{\partial x}$.
If equation~\eqref{sat_HJB} can be solved, the feedback control law is a
function defined as
\begin{equation}\label{opt_u}
  u^{*}(t,x)=u^\ast\left(t,x,V_x^\mathsf{T}(t,x)\right).
\end{equation}
The design and control of engineering systems involve both on-line and
off-line computations. In the method proposed here, the HJB equation
is solved in the off-line design phase. Once the solution of the HJB
equation is found on a sparse grid, the on-line computation for
real-time feedback control is carried out using interpolation, a
numerical process that is simple and reliable. This approach offsets
the main computational load from on-line computation to an off-line
design phase. In addition, off-line computations allow one to use more
powerful computers than those onboard systems such as satellites or
unmanned vehicles. Because of the causality-free property, which is
explained later, using parallel computers significantly reduces the
required computational time when solving a HJB equation.
In this paper, we make the following assumption.

\begin{assumption}\label{ass1}
The optimal control is uniquely determined by the PMP at each point in the state
space.
\end{assumption}
If $H(t,x,\lambda, u)$ is a strictly convex function of $(x,u)$ over
an open convex set containing all the admissible values of $(x,u)$ at
fixed $t$ and $\lambda$, Assumption~\ref{ass1} holds true.  In general,
Assumption~\ref{ass1} can be guaranteed based on necessary and sufficient
conditions of optimal control, which has a vast
literature of publications, including both classic and viscosity solutions
\cite{Fleming-Soner,bardi-capuzzo}.
For problems with a proved existence and
uniqueness of solutions, finding the optimal feedback control for practical
real-time applications is still a difficult problem. If a system has four or
more state variables, finding a desecrate solution and implementing real-time
interpolation for feedback control suffer the curse of dimensionality. The main
contribution of this paper is that, for a well-defined problem of optimal
control with a moderate dimension, it is possible to achieve optimal feedback
control using interpolation on a discrete approximate solution.

\section{The Sparse Grid Characteristics Method}
\label{sec_algorithm}

In many numerical methods for HJB equations, which are typically solved backward
in time, the discretization is based on spatial causality and the computation is
explicit in time.
The value of the solution function $V(t,x)$ at a grid point is computed at an
earlier time using the known value of the function at neighboring grid points at
a later time.
This coupling usually comes from the discretization of the spatial derivatives.
For HJB equations of high dimensions, in our examples the dimension $d=6$,
solving the equation using traditional algorithms based on dense grids is
computationally challenging.
For instance, in a six dimensional space, if $2^5=32$ grid points are used to
approximate a single variable, which is quite small, the total number of grid
points for a $6$-D problem is over $10^9$.
If $100$ points are used for a single variable, then the size of the dense grid
is $10^{12}$.
The curse of dimensionality is a bottleneck problem in solving HJB equations for
practical applications.

To mitigate the curse of dimensionality, we introduce a causality free
computational method.
It consists of two components:  (1) A solver that can find the optimal control
and the value of $V(t,x)$ at any grid point; the computation is independent of
the value of $V(t,x)$ at other points in the state space, i.e., causality free.
(2) A set of grid points, such as a sparse grid, with a reduced size to make the
problem tractable.
The causality free method introduced in this section is based on BVP solvers and
sparse grids.
The goal is to solve the problem of optimal control and its associated HJB
equation with a moderate dimension. In this paper, we use a $6$-D example to
illustrate the algorithm.

\emph{Why causality free?}
From a conventional viewpoint of computation, solving a two-point BVP is not an
efficient approach for PDEs.
However, the causality free method is perfectly parallel.
In fact, each grid point can be assigned a CPU core.
The computation at a grid point requires no communication with the computation
process at other grid points.
Although not preferred in a serial computational environment, causality free
algorithms can easily be implemented in massively parallel computers.
In addition, causality free algorithms are ideal for sparse grids in which the
space between adjacent grid points varies significantly.
The combination of sparse grids, BVP solvers, and parallel computation makes it
possible to mitigate the curse of dimensionality effectively for problems in
which $d$ is not too large.

\subsection{A causality-free method using the necessary condition of optimal control}
In contrast to traditional PDE discretizations, our proposed technique does not
discretize the HJB equations directly but instead uses the PMP to derive a set
of necessary conditions in the form of a BVP for each grid point, also known as
the method of characteristics.  As a result, the computation of the solution at
an initial point in space is independent of other points.

The optimal trajectory for the optimal feedback control law described in
Section~\ref{sec_formulation} is a solution of the two-point BVP
\begin{subequations}\label{eq:hamchar}
  \begin{alignat}{1}
    \dot{\vec{x}}(s)       &=\phantom{-}
    {\left(\pder{H}{\lambda}(s,\vec{x}(s),\vec{\lambda}(s),u^\ast(s,\vec{x}(s),\vec{\lambda}(s)))\right)}^\mathsf{T},\\
    \dot{\vec{\lambda}}(s) &=
     - {\left(\pder{H}{x}(s,\vec{x}(s),\vec{\lambda}(s),u^\ast(s,\vec{x}(s),\vec{\lambda}(s)))\right)}^\mathsf{T},\\
\dot{\vec{z}}(s)       &= L(s,\vec{x}(s),u^\ast(s,\vec{x}(s),\vec{\lambda}(s))),
  \end{alignat}
  where $t \le s \le T$ with the boundary conditions
  \begin{alignat}{1}
    \vec{x}(t) &= x, \\
    \vec{\lambda}(T) &= {\left(\der{h}{x}(\vec{x}(T))\right)}^\mathsf{T},\\
    \vec{z}(t)  &= 0.
  \end{alignat}
\end{subequations}
The optimal control and the minimum costs are
\begin{align}
\label{optuv}
u^\ast(t,x)&=u^\ast(t,x,\vec{\lambda}(t)), & V(t,x)&=\vec{z}(T)+h(\vec{x}(T)).
\end{align}

Given any grid point, $x$, we can solve the BVP~\eqref{eq:hamchar} and use the
identities~\eqref{optuv} to find the optimal control and the corresponding
minimum cost without using the value of $V(t,x)$ in any nearby points, i.e., the
computation is causality free.  Numerical algorithms for solving BVPs similar
to~\eqref{eq:hamchar} have been studied by many authors. In the examples, we
adopt an algorithm based on a four-point Lobatto~IIIa formula for our BVP
solver. This is a collocation formula and the collocation polynomial provides a
solution that is fifth-order accurate (see~\citet{kierzenka} for more detail).
We would like to point out that the computation is not limited to the
Lobatto~IIIa BVP solver. The problem of optimal control at each grid point can
be solved using any computational method.

\subsection{Sparse grids}
In the approximation of multivariable functions, sparse grid interpolation is an
approach in which the approximation is build on a subset of a dense grid with a
significantly reduced size. Sparse grids are derived from the Smolyak's
construction~\cite{smolyak}. Although the original idea was invented more than
fifty years ago, some recent work reveals potentials of its applications
\cite{kangwilcox,barthelmann,klimke,bokanowski,bungartz,xiu}. It is a known fact
that the size of sparse grids increases with the dimension, $d$, on the order of
\begin{equation*}
  O(N{(\log N)}^{d-1}),
\end{equation*}
which is in sharp contrast to the size of the corresponding dense grid
\begin{equation*}
  O(N^d).
\end{equation*}
Here $N=2^{q-d}$ where $q$ is a measurement of the level of refinement of the
sparse grid.
Obviously, the significantly reduced number of grid points has its impact on the
accuracy. For example, an upper bound of interpolation error using a classic
sparse grid satisfies
\begin{equation*}
  \norm{e}_{L^2} = O(N^{-2}{(\log N)}^{d-1})
\end{equation*}
for all functions with bounded mixed derivatives up to the second order.
Compared with the error bound using a dense grid, which is $O(N^{-2})$, we pay a
small price in terms of accuracy for problems with a moderate dimension and we
achieve a significantly reduced size of the grid.

Sparse grids have a hierarchical structure. For each variable, the set of grid
points contains several layers of subsets, denoted by $X^i$. Let $N_i$ be
the number of points in $X^i$. These subsets are assumed to have a telescope
structure, $X^{i-1} \subset X^i$ for $i> 1$. For illustration purposes, we
exemplify the definition of the classic sparse grid using equally spaced nodes
in ${[0,1]}^d$ as
\begin{alignat}{1}\label{classicalX}
  N_i=2^{i-1}+1, &\quad X^i=\Set*{\frac{k-1}{2^{i-1}} \given  k=1,2,\ldots, N_i},
\end{alignat}
for $i\ge1$.

The set of points in $X^i$ but not in $X^{i-1}$ is denoted by $\Delta X^i$; and
the number of points in the set is $\dltN_i$, i.e.,
\begin{alignat*}{2}
  \Delta X^1 &= X^1, \\
  \Delta X^i &= X^i-X^{i-1}      &\qquad (i\geq 2), \\
     \dltN_i &= \abs{\Delta X^i} &\qquad (i\geq 1).
\end{alignat*}
In this paper, the points in $X^i$ are represented by $x^i_j$,
$1\leq j \leq N_i$, and the points in $\Delta X^i$ are represented by
$\dltx^i_j$, $1\leq j \leq \dltN_i$, i.e.,
\begin{align*}
       X^i&=\Set*{x^i_1, x^i_2, \ldots, x^i_{N_i}}, \\
\Delta X^i&=\Set*{\dltx^i_1, \dltx^i_2, \ldots, \dltx^i_{\dltN_i}},
\end{align*}
for $i\ge1$.
In the sequel, we adopt the multi-index notations
\begin{align*}
         \bi  &= \begin{bmatrix} i_1&i_2&\cdots &i_d\end{bmatrix},\\
    \abs{\bi} &= i_1+i_2+\cdots+i_d,\\
    x^\bi_\bj &=(x^{i_1}_{j_1},x^{i_2}_{j_2},\ldots,x^{i_d}_{j_d}),\\
    \Delta X^{\bi}&=\Delta X^{i_1}\times \Delta X^{i_2}\times
                    \cdots \times \Delta X^{i_d}.
\end{align*}
The dense grid build on $X^q$ for an integer $q>0$, denoted as $\Gdn^q$, is
\begin{equation*}
  \Gdn^q = X^{q}\times \cdots \times X^{q} =
    \bigcup_{ 1\leq i_1,\ldots,i_d \leq q} \Delta X^{\bi}.
\end{equation*}
Following Smolyak's approximation algorithm~\cite{barthelmann,smolyak}, the
sparse grid, denoted by $\Gsp^q$, is defined as
\begin{equation*}
  \Gsp^q =\bigcup_{\abs{\bi}\leq q}\Delta X^{\bi}.
\end{equation*}
Two plots of $2$-D sparse grids are shown in Figure~\ref{fig_sparsegrid_classic}
for $q=6$ and $8$. If $q=8$, $\Gsp^q$ has a total of $385$ grid points whereas
the corresponding dense grid, $\Gdn^q$, has ${(2^6+1)}^2=4225$ points. The difference
becomes increasingly significant for higher dimensions.
\begin{figure}
  \centering
\begin{tikzpicture}

\begin{axis}[%
width=\dotfigurewidth,
height=\dotfigureheight,
scale only axis,
ticks=none,
separate axis lines,
xmin=0,
xmax=1,
ymin=0,
ymax=1
]
\addplot [only marks, mark=*, mark size=\dotsize, color=dotcolor,forget plot]
  table[row sep=crcr]{0	0\\
0	1\\
1	0\\
1	1\\
0	0.5\\
1	0.5\\
0	0.25\\
0	0.75\\
1	0.25\\
1	0.75\\
0	0.125\\
0	0.375\\
0	0.625\\
0	0.875\\
1	0.125\\
1	0.375\\
1	0.625\\
1	0.875\\
0	0.0625\\
0	0.1875\\
0	0.3125\\
0	0.4375\\
0	0.5625\\
0	0.6875\\
0	0.8125\\
0	0.9375\\
1	0.0625\\
1	0.1875\\
1	0.3125\\
1	0.4375\\
1	0.5625\\
1	0.6875\\
1	0.8125\\
1	0.9375\\
0.5	0\\
0.5	1\\
0.5	0.5\\
0.5	0.25\\
0.5	0.75\\
0.5	0.125\\
0.5	0.375\\
0.5	0.625\\
0.5	0.875\\
0.25	0\\
0.25	1\\
0.75	0\\
0.75	1\\
0.25	0.5\\
0.75	0.5\\
0.25	0.25\\
0.25	0.75\\
0.75	0.25\\
0.75	0.75\\
0.125	0\\
0.125	1\\
0.375	0\\
0.375	1\\
0.625	0\\
0.625	1\\
0.875	0\\
0.875	1\\
0.125	0.5\\
0.375	0.5\\
0.625	0.5\\
0.875	0.5\\
0.0625	0\\
0.0625	1\\
0.1875	0\\
0.1875	1\\
0.3125	0\\
0.3125	1\\
0.4375	0\\
0.4375	1\\
0.5625	0\\
0.5625	1\\
0.6875	0\\
0.6875	1\\
0.8125	0\\
0.8125	1\\
0.9375	0\\
0.9375	1\\
};
\end{axis}
\end{tikzpicture}%
   \qquad
\begin{tikzpicture}

\begin{axis}[%
width=\dotfigurewidth,
height=\dotfigureheight,
ticks=none,
scale only axis,
separate axis lines,
xmin=0,
xmax=1,
ymin=0,
ymax=1
]
\addplot [color=dotcolor,only marks,mark=*, mark size=\dotsize,forget plot]
  table[row sep=crcr]{0	0\\
0	1\\
1	0\\
1	1\\
0	0.5\\
1	0.5\\
0	0.25\\
0	0.75\\
1	0.25\\
1	0.75\\
0	0.125\\
0	0.375\\
0	0.625\\
0	0.875\\
1	0.125\\
1	0.375\\
1	0.625\\
1	0.875\\
0	0.0625\\
0	0.1875\\
0	0.3125\\
0	0.4375\\
0	0.5625\\
0	0.6875\\
0	0.8125\\
0	0.9375\\
1	0.0625\\
1	0.1875\\
1	0.3125\\
1	0.4375\\
1	0.5625\\
1	0.6875\\
1	0.8125\\
1	0.9375\\
0	0.03125\\
0	0.09375\\
0	0.15625\\
0	0.21875\\
0	0.28125\\
0	0.34375\\
0	0.40625\\
0	0.46875\\
0	0.53125\\
0	0.59375\\
0	0.65625\\
0	0.71875\\
0	0.78125\\
0	0.84375\\
0	0.90625\\
0	0.96875\\
1	0.03125\\
1	0.09375\\
1	0.15625\\
1	0.21875\\
1	0.28125\\
1	0.34375\\
1	0.40625\\
1	0.46875\\
1	0.53125\\
1	0.59375\\
1	0.65625\\
1	0.71875\\
1	0.78125\\
1	0.84375\\
1	0.90625\\
1	0.96875\\
0	0.015625\\
0	0.046875\\
0	0.078125\\
0	0.109375\\
0	0.140625\\
0	0.171875\\
0	0.203125\\
0	0.234375\\
0	0.265625\\
0	0.296875\\
0	0.328125\\
0	0.359375\\
0	0.390625\\
0	0.421875\\
0	0.453125\\
0	0.484375\\
0	0.515625\\
0	0.546875\\
0	0.578125\\
0	0.609375\\
0	0.640625\\
0	0.671875\\
0	0.703125\\
0	0.734375\\
0	0.765625\\
0	0.796875\\
0	0.828125\\
0	0.859375\\
0	0.890625\\
0	0.921875\\
0	0.953125\\
0	0.984375\\
1	0.015625\\
1	0.046875\\
1	0.078125\\
1	0.109375\\
1	0.140625\\
1	0.171875\\
1	0.203125\\
1	0.234375\\
1	0.265625\\
1	0.296875\\
1	0.328125\\
1	0.359375\\
1	0.390625\\
1	0.421875\\
1	0.453125\\
1	0.484375\\
1	0.515625\\
1	0.546875\\
1	0.578125\\
1	0.609375\\
1	0.640625\\
1	0.671875\\
1	0.703125\\
1	0.734375\\
1	0.765625\\
1	0.796875\\
1	0.828125\\
1	0.859375\\
1	0.890625\\
1	0.921875\\
1	0.953125\\
1	0.984375\\
0.5	0\\
0.5	1\\
0.5	0.5\\
0.5	0.25\\
0.5	0.75\\
0.5	0.125\\
0.5	0.375\\
0.5	0.625\\
0.5	0.875\\
0.5	0.0625\\
0.5	0.1875\\
0.5	0.3125\\
0.5	0.4375\\
0.5	0.5625\\
0.5	0.6875\\
0.5	0.8125\\
0.5	0.9375\\
0.5	0.03125\\
0.5	0.09375\\
0.5	0.15625\\
0.5	0.21875\\
0.5	0.28125\\
0.5	0.34375\\
0.5	0.40625\\
0.5	0.46875\\
0.5	0.53125\\
0.5	0.59375\\
0.5	0.65625\\
0.5	0.71875\\
0.5	0.78125\\
0.5	0.84375\\
0.5	0.90625\\
0.5	0.96875\\
0.25	0\\
0.25	1\\
0.75	0\\
0.75	1\\
0.25	0.5\\
0.75	0.5\\
0.25	0.25\\
0.25	0.75\\
0.75	0.25\\
0.75	0.75\\
0.25	0.125\\
0.25	0.375\\
0.25	0.625\\
0.25	0.875\\
0.75	0.125\\
0.75	0.375\\
0.75	0.625\\
0.75	0.875\\
0.25	0.0625\\
0.25	0.1875\\
0.25	0.3125\\
0.25	0.4375\\
0.25	0.5625\\
0.25	0.6875\\
0.25	0.8125\\
0.25	0.9375\\
0.75	0.0625\\
0.75	0.1875\\
0.75	0.3125\\
0.75	0.4375\\
0.75	0.5625\\
0.75	0.6875\\
0.75	0.8125\\
0.75	0.9375\\
0.125	0\\
0.125	1\\
0.375	0\\
0.375	1\\
0.625	0\\
0.625	1\\
0.875	0\\
0.875	1\\
0.125	0.5\\
0.375	0.5\\
0.625	0.5\\
0.875	0.5\\
0.125	0.25\\
0.125	0.75\\
0.375	0.25\\
0.375	0.75\\
0.625	0.25\\
0.625	0.75\\
0.875	0.25\\
0.875	0.75\\
0.125	0.125\\
0.125	0.375\\
0.125	0.625\\
0.125	0.875\\
0.375	0.125\\
0.375	0.375\\
0.375	0.625\\
0.375	0.875\\
0.625	0.125\\
0.625	0.375\\
0.625	0.625\\
0.625	0.875\\
0.875	0.125\\
0.875	0.375\\
0.875	0.625\\
0.875	0.875\\
0.0625	0\\
0.0625	1\\
0.1875	0\\
0.1875	1\\
0.3125	0\\
0.3125	1\\
0.4375	0\\
0.4375	1\\
0.5625	0\\
0.5625	1\\
0.6875	0\\
0.6875	1\\
0.8125	0\\
0.8125	1\\
0.9375	0\\
0.9375	1\\
0.0625	0.5\\
0.1875	0.5\\
0.3125	0.5\\
0.4375	0.5\\
0.5625	0.5\\
0.6875	0.5\\
0.8125	0.5\\
0.9375	0.5\\
0.0625	0.25\\
0.0625	0.75\\
0.1875	0.25\\
0.1875	0.75\\
0.3125	0.25\\
0.3125	0.75\\
0.4375	0.25\\
0.4375	0.75\\
0.5625	0.25\\
0.5625	0.75\\
0.6875	0.25\\
0.6875	0.75\\
0.8125	0.25\\
0.8125	0.75\\
0.9375	0.25\\
0.9375	0.75\\
0.03125	0\\
0.03125	1\\
0.09375	0\\
0.09375	1\\
0.15625	0\\
0.15625	1\\
0.21875	0\\
0.21875	1\\
0.28125	0\\
0.28125	1\\
0.34375	0\\
0.34375	1\\
0.40625	0\\
0.40625	1\\
0.46875	0\\
0.46875	1\\
0.53125	0\\
0.53125	1\\
0.59375	0\\
0.59375	1\\
0.65625	0\\
0.65625	1\\
0.71875	0\\
0.71875	1\\
0.78125	0\\
0.78125	1\\
0.84375	0\\
0.84375	1\\
0.90625	0\\
0.90625	1\\
0.96875	0\\
0.96875	1\\
0.03125	0.5\\
0.09375	0.5\\
0.15625	0.5\\
0.21875	0.5\\
0.28125	0.5\\
0.34375	0.5\\
0.40625	0.5\\
0.46875	0.5\\
0.53125	0.5\\
0.59375	0.5\\
0.65625	0.5\\
0.71875	0.5\\
0.78125	0.5\\
0.84375	0.5\\
0.90625	0.5\\
0.96875	0.5\\
0.015625	0\\
0.015625	1\\
0.046875	0\\
0.046875	1\\
0.078125	0\\
0.078125	1\\
0.109375	0\\
0.109375	1\\
0.140625	0\\
0.140625	1\\
0.171875	0\\
0.171875	1\\
0.203125	0\\
0.203125	1\\
0.234375	0\\
0.234375	1\\
0.265625	0\\
0.265625	1\\
0.296875	0\\
0.296875	1\\
0.328125	0\\
0.328125	1\\
0.359375	0\\
0.359375	1\\
0.390625	0\\
0.390625	1\\
0.421875	0\\
0.421875	1\\
0.453125	0\\
0.453125	1\\
0.484375	0\\
0.484375	1\\
0.515625	0\\
0.515625	1\\
0.546875	0\\
0.546875	1\\
0.578125	0\\
0.578125	1\\
0.609375	0\\
0.609375	1\\
0.640625	0\\
0.640625	1\\
0.671875	0\\
0.671875	1\\
0.703125	0\\
0.703125	1\\
0.734375	0\\
0.734375	1\\
0.765625	0\\
0.765625	1\\
0.796875	0\\
0.796875	1\\
0.828125	0\\
0.828125	1\\
0.859375	0\\
0.859375	1\\
0.890625	0\\
0.890625	1\\
0.921875	0\\
0.921875	1\\
0.953125	0\\
0.953125	1\\
0.984375	0\\
0.984375	1\\
};
\end{axis}
\end{tikzpicture}%
   \caption{The classic sparse grid in ${[0, 1]}^2$, $q=6$ and
    $q=8$}\label{fig_sparsegrid_classic}
\end{figure}
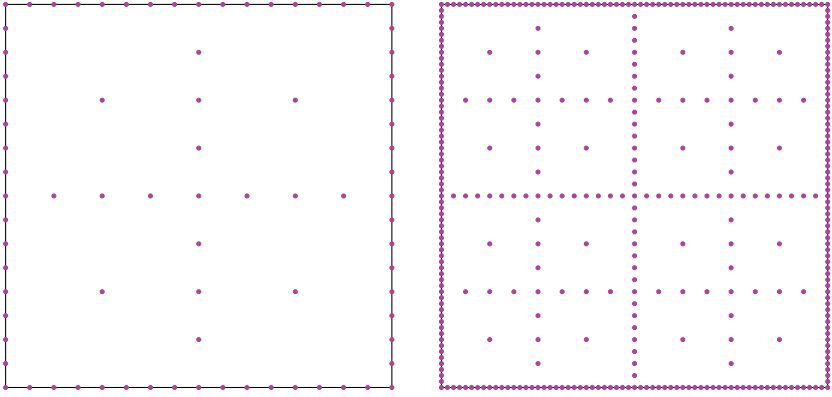

Sparse grids can be build using other nested sequences of grids $X^i$, $i\geq
1$.  For instance, a modified sparse grid is defined using
\begin{equation}\label{modified}
  \begin{dcases}
    N_1=1,         & X^1=\Set*{\frac{1}{2}};\\
    N_i=2^{i-1}+1, & X^i=\Set*{\frac{k-1}{2^{i-1}}\given k=1,2,\ldots,N_i};
  \end{dcases}
\end{equation}
for $i>1$.  Except for $i=1$, it is identical to the classic sparse
grid~\eqref{classicalX}.  The size of $\Gsp$ is further reduced. The modified
sparse grids for $q=6$ and $q=8$ are shown in
Figure~\ref{fig_sparsegrid_modified}. When $q=8$, the total number of points is
$321$.
\begin{figure}
  \centering
\begin{tikzpicture}

\begin{axis}[%
width=\dotfigurewidth,
height=\dotfigureheight,
ticks=none,
scale only axis,
separate axis lines,
xmin=0,
xmax=1,
ymin=0,
ymax=1
]
\addplot [only marks, mark=*, mark size=\dotsize, color=dotcolor,forget plot]
  table[row sep=crcr]{0.5	0.5\\
0	0.5\\
1	0.5\\
0.5	0\\
0.5	1\\
0.25	0.5\\
0.75	0.5\\
0	0\\
1	0\\
0	1\\
1	1\\
0.5	0.25\\
0.5	0.75\\
0.125	0.5\\
0.375	0.5\\
0.625	0.5\\
0.875	0.5\\
0.25	0\\
0.75	0\\
0.25	1\\
0.75	1\\
0	0.25\\
1	0.25\\
0	0.75\\
1	0.75\\
0.5	0.125\\
0.5	0.375\\
0.5	0.625\\
0.5	0.875\\
0.0625	0.5\\
0.1875	0.5\\
0.3125	0.5\\
0.4375	0.5\\
0.5625	0.5\\
0.6875	0.5\\
0.8125	0.5\\
0.9375	0.5\\
0.125	0\\
0.375	0\\
0.625	0\\
0.875	0\\
0.125	1\\
0.375	1\\
0.625	1\\
0.875	1\\
0.25	0.25\\
0.75	0.25\\
0.25	0.75\\
0.75	0.75\\
0	0.125\\
1	0.125\\
0	0.375\\
1	0.375\\
0	0.625\\
1	0.625\\
0	0.875\\
1	0.875\\
0.5	0.0625\\
0.5	0.1875\\
0.5	0.3125\\
0.5	0.4375\\
0.5	0.5625\\
0.5	0.6875\\
0.5	0.8125\\
0.5	0.9375\\
};
\end{axis}
\end{tikzpicture}%
   \qquad
\begin{tikzpicture}

\begin{axis}[%
width=\dotfigurewidth,
height=\dotfigureheight,
ticks=none,
scale only axis,
separate axis lines,
xmin=0,
xmax=1,
ymin=0,
ymax=1
]
\addplot [only marks, mark=*, mark size=\dotsize, color=dotcolor,forget plot]
  table[row sep=crcr]{0.5	0.5\\
0	0.5\\
1	0.5\\
0.5	0\\
0.5	1\\
0.25	0.5\\
0.75	0.5\\
0	0\\
1	0\\
0	1\\
1	1\\
0.5	0.25\\
0.5	0.75\\
0.125	0.5\\
0.375	0.5\\
0.625	0.5\\
0.875	0.5\\
0.25	0\\
0.75	0\\
0.25	1\\
0.75	1\\
0	0.25\\
1	0.25\\
0	0.75\\
1	0.75\\
0.5	0.125\\
0.5	0.375\\
0.5	0.625\\
0.5	0.875\\
0.0625	0.5\\
0.1875	0.5\\
0.3125	0.5\\
0.4375	0.5\\
0.5625	0.5\\
0.6875	0.5\\
0.8125	0.5\\
0.9375	0.5\\
0.125	0\\
0.375	0\\
0.625	0\\
0.875	0\\
0.125	1\\
0.375	1\\
0.625	1\\
0.875	1\\
0.25	0.25\\
0.75	0.25\\
0.25	0.75\\
0.75	0.75\\
0	0.125\\
1	0.125\\
0	0.375\\
1	0.375\\
0	0.625\\
1	0.625\\
0	0.875\\
1	0.875\\
0.5	0.0625\\
0.5	0.1875\\
0.5	0.3125\\
0.5	0.4375\\
0.5	0.5625\\
0.5	0.6875\\
0.5	0.8125\\
0.5	0.9375\\
0.03125	0.5\\
0.09375	0.5\\
0.15625	0.5\\
0.21875	0.5\\
0.28125	0.5\\
0.34375	0.5\\
0.40625	0.5\\
0.46875	0.5\\
0.53125	0.5\\
0.59375	0.5\\
0.65625	0.5\\
0.71875	0.5\\
0.78125	0.5\\
0.84375	0.5\\
0.90625	0.5\\
0.96875	0.5\\
0.0625	0\\
0.1875	0\\
0.3125	0\\
0.4375	0\\
0.5625	0\\
0.6875	0\\
0.8125	0\\
0.9375	0\\
0.0625	1\\
0.1875	1\\
0.3125	1\\
0.4375	1\\
0.5625	1\\
0.6875	1\\
0.8125	1\\
0.9375	1\\
0.125	0.25\\
0.375	0.25\\
0.625	0.25\\
0.875	0.25\\
0.125	0.75\\
0.375	0.75\\
0.625	0.75\\
0.875	0.75\\
0.25	0.125\\
0.75	0.125\\
0.25	0.375\\
0.75	0.375\\
0.25	0.625\\
0.75	0.625\\
0.25	0.875\\
0.75	0.875\\
0	0.0625\\
1	0.0625\\
0	0.1875\\
1	0.1875\\
0	0.3125\\
1	0.3125\\
0	0.4375\\
1	0.4375\\
0	0.5625\\
1	0.5625\\
0	0.6875\\
1	0.6875\\
0	0.8125\\
1	0.8125\\
0	0.9375\\
1	0.9375\\
0.5	0.03125\\
0.5	0.09375\\
0.5	0.15625\\
0.5	0.21875\\
0.5	0.28125\\
0.5	0.34375\\
0.5	0.40625\\
0.5	0.46875\\
0.5	0.53125\\
0.5	0.59375\\
0.5	0.65625\\
0.5	0.71875\\
0.5	0.78125\\
0.5	0.84375\\
0.5	0.90625\\
0.5	0.96875\\
0.015625	0.5\\
0.046875	0.5\\
0.078125	0.5\\
0.109375	0.5\\
0.140625	0.5\\
0.171875	0.5\\
0.203125	0.5\\
0.234375	0.5\\
0.265625	0.5\\
0.296875	0.5\\
0.328125	0.5\\
0.359375	0.5\\
0.390625	0.5\\
0.421875	0.5\\
0.453125	0.5\\
0.484375	0.5\\
0.515625	0.5\\
0.546875	0.5\\
0.578125	0.5\\
0.609375	0.5\\
0.640625	0.5\\
0.671875	0.5\\
0.703125	0.5\\
0.734375	0.5\\
0.765625	0.5\\
0.796875	0.5\\
0.828125	0.5\\
0.859375	0.5\\
0.890625	0.5\\
0.921875	0.5\\
0.953125	0.5\\
0.984375	0.5\\
0.03125	0\\
0.09375	0\\
0.15625	0\\
0.21875	0\\
0.28125	0\\
0.34375	0\\
0.40625	0\\
0.46875	0\\
0.53125	0\\
0.59375	0\\
0.65625	0\\
0.71875	0\\
0.78125	0\\
0.84375	0\\
0.90625	0\\
0.96875	0\\
0.03125	1\\
0.09375	1\\
0.15625	1\\
0.21875	1\\
0.28125	1\\
0.34375	1\\
0.40625	1\\
0.46875	1\\
0.53125	1\\
0.59375	1\\
0.65625	1\\
0.71875	1\\
0.78125	1\\
0.84375	1\\
0.90625	1\\
0.96875	1\\
0.0625	0.25\\
0.1875	0.25\\
0.3125	0.25\\
0.4375	0.25\\
0.5625	0.25\\
0.6875	0.25\\
0.8125	0.25\\
0.9375	0.25\\
0.0625	0.75\\
0.1875	0.75\\
0.3125	0.75\\
0.4375	0.75\\
0.5625	0.75\\
0.6875	0.75\\
0.8125	0.75\\
0.9375	0.75\\
0.125	0.125\\
0.375	0.125\\
0.625	0.125\\
0.875	0.125\\
0.125	0.375\\
0.375	0.375\\
0.625	0.375\\
0.875	0.375\\
0.125	0.625\\
0.375	0.625\\
0.625	0.625\\
0.875	0.625\\
0.125	0.875\\
0.375	0.875\\
0.625	0.875\\
0.875	0.875\\
0.25	0.0625\\
0.75	0.0625\\
0.25	0.1875\\
0.75	0.1875\\
0.25	0.3125\\
0.75	0.3125\\
0.25	0.4375\\
0.75	0.4375\\
0.25	0.5625\\
0.75	0.5625\\
0.25	0.6875\\
0.75	0.6875\\
0.25	0.8125\\
0.75	0.8125\\
0.25	0.9375\\
0.75	0.9375\\
0	0.03125\\
1	0.03125\\
0	0.09375\\
1	0.09375\\
0	0.15625\\
1	0.15625\\
0	0.21875\\
1	0.21875\\
0	0.28125\\
1	0.28125\\
0	0.34375\\
1	0.34375\\
0	0.40625\\
1	0.40625\\
0	0.46875\\
1	0.46875\\
0	0.53125\\
1	0.53125\\
0	0.59375\\
1	0.59375\\
0	0.65625\\
1	0.65625\\
0	0.71875\\
1	0.71875\\
0	0.78125\\
1	0.78125\\
0	0.84375\\
1	0.84375\\
0	0.90625\\
1	0.90625\\
0	0.96875\\
1	0.96875\\
0.5	0.015625\\
0.5	0.046875\\
0.5	0.078125\\
0.5	0.109375\\
0.5	0.140625\\
0.5	0.171875\\
0.5	0.203125\\
0.5	0.234375\\
0.5	0.265625\\
0.5	0.296875\\
0.5	0.328125\\
0.5	0.359375\\
0.5	0.390625\\
0.5	0.421875\\
0.5	0.453125\\
0.5	0.484375\\
0.5	0.515625\\
0.5	0.546875\\
0.5	0.578125\\
0.5	0.609375\\
0.5	0.640625\\
0.5	0.671875\\
0.5	0.703125\\
0.5	0.734375\\
0.5	0.765625\\
0.5	0.796875\\
0.5	0.828125\\
0.5	0.859375\\
0.5	0.890625\\
0.5	0.921875\\
0.5	0.953125\\
0.5	0.984375\\
};
\end{axis}
\end{tikzpicture}%
   \caption{Modified sparse grids in ${[0, 1]}^2$, $q=6$ and
    $q=8$}\label{fig_sparsegrid_modified}
\end{figure}
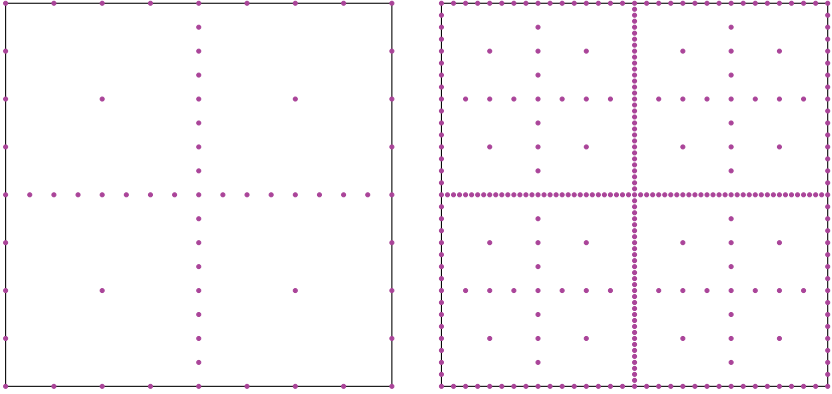

The Chebyshev Gauss--Lobatto (CGL) sparse grid is defined in a similar way.  The
the grid points are defined as
\begin{equation}\label{CGL}
  \begin{dcases}
    N_1=1,         & X^1=\Set*{\frac{1}{2}};\\
    N_i=2^{i-1}+1, & X^i=\Set*{\frac{1}{2}
                               \left(1-\cos \frac{(k-1)\pi}{2^{i-1}}\right)\given
                               k=1,2,\ldots,N_i};
  \end{dcases}
\end{equation}
for $i>1$.  Two examples of this grid are shown in
Figure~\ref{fig_sparsegrid_CGL}.
\begin{figure}
  \centering
\begin{tikzpicture}

\begin{axis}[%
width=\dotfigurewidth,
height=\dotfigureheight,
ticks=none,
scale only axis,
separate axis lines,
xmin=-1,
xmax=1,
ymin=-1,
ymax=1
]
\addplot [only marks, mark=*, mark size=\dotsize, color=dotcolor,forget plot]
  table[row sep=crcr]{-1	-1\\
-1	-0.923879532511287\\
-1	-0.707106781186548\\
-1	-0.38268343236509\\
-1	0\\
-1	0.38268343236509\\
-1	0.707106781186548\\
-1	0.923879532511287\\
-1	1\\
-0.98078528040323	0\\
-0.923879532511287	-1\\
-0.923879532511287	0\\
-0.923879532511287	1\\
-0.831469612302545	0\\
-0.707106781186548	-1\\
-0.707106781186548	-0.707106781186548\\
-0.707106781186548	0\\
-0.707106781186548	0.707106781186548\\
-0.707106781186548	1\\
-0.555570233019602	0\\
-0.38268343236509	-1\\
-0.38268343236509	0\\
-0.38268343236509	1\\
-0.195090322016128	0\\
0	-1\\
0	-0.98078528040323\\
0	-0.923879532511287\\
0	-0.831469612302545\\
0	-0.707106781186548\\
0	-0.555570233019602\\
0	-0.38268343236509\\
0	-0.195090322016128\\
0	0\\
0	0.195090322016128\\
0	0.38268343236509\\
0	0.555570233019602\\
0	0.707106781186548\\
0	0.831469612302545\\
0	0.923879532511287\\
0	0.98078528040323\\
0	1\\
0.195090322016128	0\\
0.38268343236509	-1\\
0.38268343236509	0\\
0.38268343236509	1\\
0.555570233019602	0\\
0.707106781186548	-1\\
0.707106781186548	-0.707106781186548\\
0.707106781186548	0\\
0.707106781186548	0.707106781186548\\
0.707106781186548	1\\
0.831469612302545	0\\
0.923879532511287	-1\\
0.923879532511287	0\\
0.923879532511287	1\\
0.98078528040323	0\\
1	-1\\
1	-0.923879532511287\\
1	-0.707106781186548\\
1	-0.38268343236509\\
1	0\\
1	0.38268343236509\\
1	0.707106781186548\\
1	0.923879532511287\\
1	1\\
};
\end{axis}
\end{tikzpicture}%
   \qquad
\begin{tikzpicture}

\begin{axis}[%
width=\dotfigurewidth,
height=\dotfigureheight,
ticks=none,
scale only axis,
separate axis lines,
xmin=-1,
xmax=1,
ymin=-1,
ymax=1
]
\addplot [only marks, mark=*, mark size=\dotsize, color=dotcolor,forget plot]
  table[row sep=crcr]{-1	-1\\
-1	-0.995184726672197\\
-1	-0.98078528040323\\
-1	-0.956940335732209\\
-1	-0.923879532511287\\
-1	-0.881921264348355\\
-1	-0.831469612302545\\
-1	-0.773010453362737\\
-1	-0.707106781186548\\
-1	-0.634393284163645\\
-1	-0.555570233019602\\
-1	-0.471396736825998\\
-1	-0.38268343236509\\
-1	-0.290284677254462\\
-1	-0.195090322016128\\
-1	-0.098017140329561\\
-1	0\\
-1	0.098017140329561\\
-1	0.195090322016128\\
-1	0.290284677254462\\
-1	0.38268343236509\\
-1	0.471396736825998\\
-1	0.555570233019602\\
-1	0.634393284163646\\
-1	0.707106781186548\\
-1	0.773010453362737\\
-1	0.831469612302545\\
-1	0.881921264348355\\
-1	0.923879532511287\\
-1	0.956940335732209\\
-1	0.98078528040323\\
-1	0.995184726672197\\
-1	1\\
-0.998795456205172	0\\
-0.995184726672197	-1\\
-0.995184726672197	0\\
-0.995184726672197	1\\
-0.989176509964781	0\\
-0.98078528040323	-1\\
-0.98078528040323	-0.707106781186548\\
-0.98078528040323	0\\
-0.98078528040323	0.707106781186548\\
-0.98078528040323	1\\
-0.970031253194544	0\\
-0.956940335732209	-1\\
-0.956940335732209	0\\
-0.956940335732209	1\\
-0.941544065183021	0\\
-0.923879532511287	-1\\
-0.923879532511287	-0.923879532511287\\
-0.923879532511287	-0.707106781186548\\
-0.923879532511287	-0.38268343236509\\
-0.923879532511287	0\\
-0.923879532511287	0.38268343236509\\
-0.923879532511287	0.707106781186548\\
-0.923879532511287	0.923879532511287\\
-0.923879532511287	1\\
-0.903989293123443	0\\
-0.881921264348355	-1\\
-0.881921264348355	0\\
-0.881921264348355	1\\
-0.857728610000272	0\\
-0.831469612302545	-1\\
-0.831469612302545	-0.707106781186548\\
-0.831469612302545	0\\
-0.831469612302545	0.707106781186548\\
-0.831469612302545	1\\
-0.803207531480645	0\\
-0.773010453362737	-1\\
-0.773010453362737	0\\
-0.773010453362737	1\\
-0.740951125354959	0\\
-0.707106781186548	-1\\
-0.707106781186548	-0.98078528040323\\
-0.707106781186548	-0.923879532511287\\
-0.707106781186548	-0.831469612302545\\
-0.707106781186548	-0.707106781186548\\
-0.707106781186548	-0.555570233019602\\
-0.707106781186548	-0.38268343236509\\
-0.707106781186548	-0.195090322016128\\
-0.707106781186548	0\\
-0.707106781186548	0.195090322016128\\
-0.707106781186548	0.38268343236509\\
-0.707106781186548	0.555570233019602\\
-0.707106781186548	0.707106781186548\\
-0.707106781186548	0.831469612302545\\
-0.707106781186548	0.923879532511287\\
-0.707106781186548	0.98078528040323\\
-0.707106781186548	1\\
-0.671558954847019	0\\
-0.634393284163645	-1\\
-0.634393284163645	0\\
-0.634393284163645	1\\
-0.595699304492433	0\\
-0.555570233019602	-1\\
-0.555570233019602	-0.707106781186548\\
-0.555570233019602	0\\
-0.555570233019602	0.707106781186548\\
-0.555570233019602	1\\
-0.514102744193222	0\\
-0.471396736825998	-1\\
-0.471396736825998	0\\
-0.471396736825998	1\\
-0.427555093430282	0\\
-0.38268343236509	-1\\
-0.38268343236509	-0.923879532511287\\
-0.38268343236509	-0.707106781186548\\
-0.38268343236509	-0.38268343236509\\
-0.38268343236509	0\\
-0.38268343236509	0.38268343236509\\
-0.38268343236509	0.707106781186548\\
-0.38268343236509	0.923879532511287\\
-0.38268343236509	1\\
-0.33688985339222	0\\
-0.290284677254462	-1\\
-0.290284677254462	0\\
-0.290284677254462	1\\
-0.242980179903264	0\\
-0.195090322016128	-1\\
-0.195090322016128	-0.707106781186548\\
-0.195090322016128	0\\
-0.195090322016128	0.707106781186548\\
-0.195090322016128	1\\
-0.146730474455362	0\\
-0.098017140329561	-1\\
-0.098017140329561	0\\
-0.098017140329561	1\\
-0.049067674327418	0\\
0	-1\\
0	-0.998795456205172\\
0	-0.995184726672197\\
0	-0.989176509964781\\
0	-0.98078528040323\\
0	-0.970031253194544\\
0	-0.956940335732209\\
0	-0.941544065183021\\
0	-0.923879532511287\\
0	-0.903989293123443\\
0	-0.881921264348355\\
0	-0.857728610000272\\
0	-0.831469612302545\\
0	-0.803207531480645\\
0	-0.773010453362737\\
0	-0.740951125354959\\
0	-0.707106781186548\\
0	-0.671558954847019\\
0	-0.634393284163645\\
0	-0.595699304492433\\
0	-0.555570233019602\\
0	-0.514102744193222\\
0	-0.471396736825998\\
0	-0.427555093430282\\
0	-0.38268343236509\\
0	-0.33688985339222\\
0	-0.290284677254462\\
0	-0.242980179903264\\
0	-0.195090322016128\\
0	-0.146730474455362\\
0	-0.098017140329561\\
0	-0.049067674327418\\
0	0\\
0	0.049067674327418\\
0	0.098017140329561\\
0	0.146730474455362\\
0	0.195090322016128\\
0	0.242980179903264\\
0	0.290284677254462\\
0	0.33688985339222\\
0	0.38268343236509\\
0	0.427555093430282\\
0	0.471396736825998\\
0	0.514102744193222\\
0	0.555570233019602\\
0	0.595699304492434\\
0	0.634393284163646\\
0	0.671558954847019\\
0	0.707106781186548\\
0	0.740951125354959\\
0	0.773010453362737\\
0	0.803207531480645\\
0	0.831469612302545\\
0	0.857728610000272\\
0	0.881921264348355\\
0	0.903989293123443\\
0	0.923879532511287\\
0	0.941544065183021\\
0	0.956940335732209\\
0	0.970031253194544\\
0	0.98078528040323\\
0	0.989176509964781\\
0	0.995184726672197\\
0	0.998795456205172\\
0	1\\
0.049067674327418	0\\
0.098017140329561	-1\\
0.098017140329561	0\\
0.098017140329561	1\\
0.146730474455362	0\\
0.195090322016128	-1\\
0.195090322016128	-0.707106781186548\\
0.195090322016128	0\\
0.195090322016128	0.707106781186548\\
0.195090322016128	1\\
0.242980179903264	0\\
0.290284677254462	-1\\
0.290284677254462	0\\
0.290284677254462	1\\
0.33688985339222	0\\
0.38268343236509	-1\\
0.38268343236509	-0.923879532511287\\
0.38268343236509	-0.707106781186548\\
0.38268343236509	-0.38268343236509\\
0.38268343236509	0\\
0.38268343236509	0.38268343236509\\
0.38268343236509	0.707106781186548\\
0.38268343236509	0.923879532511287\\
0.38268343236509	1\\
0.427555093430282	0\\
0.471396736825998	-1\\
0.471396736825998	0\\
0.471396736825998	1\\
0.514102744193222	0\\
0.555570233019602	-1\\
0.555570233019602	-0.707106781186548\\
0.555570233019602	0\\
0.555570233019602	0.707106781186548\\
0.555570233019602	1\\
0.595699304492434	0\\
0.634393284163646	-1\\
0.634393284163646	0\\
0.634393284163646	1\\
0.671558954847019	0\\
0.707106781186548	-1\\
0.707106781186548	-0.98078528040323\\
0.707106781186548	-0.923879532511287\\
0.707106781186548	-0.831469612302545\\
0.707106781186548	-0.707106781186548\\
0.707106781186548	-0.555570233019602\\
0.707106781186548	-0.38268343236509\\
0.707106781186548	-0.195090322016128\\
0.707106781186548	0\\
0.707106781186548	0.195090322016128\\
0.707106781186548	0.38268343236509\\
0.707106781186548	0.555570233019602\\
0.707106781186548	0.707106781186548\\
0.707106781186548	0.831469612302545\\
0.707106781186548	0.923879532511287\\
0.707106781186548	0.98078528040323\\
0.707106781186548	1\\
0.740951125354959	0\\
0.773010453362737	-1\\
0.773010453362737	0\\
0.773010453362737	1\\
0.803207531480645	0\\
0.831469612302545	-1\\
0.831469612302545	-0.707106781186548\\
0.831469612302545	0\\
0.831469612302545	0.707106781186548\\
0.831469612302545	1\\
0.857728610000272	0\\
0.881921264348355	-1\\
0.881921264348355	0\\
0.881921264348355	1\\
0.903989293123443	0\\
0.923879532511287	-1\\
0.923879532511287	-0.923879532511287\\
0.923879532511287	-0.707106781186548\\
0.923879532511287	-0.38268343236509\\
0.923879532511287	0\\
0.923879532511287	0.38268343236509\\
0.923879532511287	0.707106781186548\\
0.923879532511287	0.923879532511287\\
0.923879532511287	1\\
0.941544065183021	0\\
0.956940335732209	-1\\
0.956940335732209	0\\
0.956940335732209	1\\
0.970031253194544	0\\
0.98078528040323	-1\\
0.98078528040323	-0.707106781186548\\
0.98078528040323	0\\
0.98078528040323	0.707106781186548\\
0.98078528040323	1\\
0.989176509964781	0\\
0.995184726672197	-1\\
0.995184726672197	0\\
0.995184726672197	1\\
0.998795456205172	0\\
1	-1\\
1	-0.995184726672197\\
1	-0.98078528040323\\
1	-0.956940335732209\\
1	-0.923879532511287\\
1	-0.881921264348355\\
1	-0.831469612302545\\
1	-0.773010453362737\\
1	-0.707106781186548\\
1	-0.634393284163645\\
1	-0.555570233019602\\
1	-0.471396736825998\\
1	-0.38268343236509\\
1	-0.290284677254462\\
1	-0.195090322016128\\
1	-0.098017140329561\\
1	0\\
1	0.098017140329561\\
1	0.195090322016128\\
1	0.290284677254462\\
1	0.38268343236509\\
1	0.471396736825998\\
1	0.555570233019602\\
1	0.634393284163646\\
1	0.707106781186548\\
1	0.773010453362737\\
1	0.831469612302545\\
1	0.881921264348355\\
1	0.923879532511287\\
1	0.956940335732209\\
1	0.98078528040323\\
1	0.995184726672197\\
1	1\\
};
\end{axis}
\end{tikzpicture}%
   \caption{CGL sparse grid in ${[0, 1]}^2$, $q=6$ and
    $q=8$}\label{fig_sparsegrid_CGL}
\end{figure}
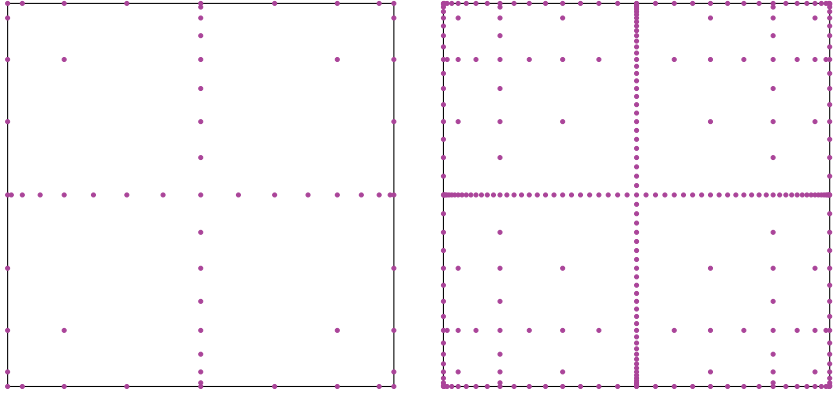

\subsection{Interpolation on sparse grids}
Given a problem of optimal control, its solution on a sparse grid can be
computed off-line. Then in real-time feedback control, the value of minimum
cost, $V(t,x)$, and the costate, $\vec{\lambda}(t)$, can be computed using
interpolation on the sparse grid.
For simplicity, this section will focus on interpolating $V(t,x)$.
Consider $X^i\subseteq [0, 1]$, $i\geq 1$. A basis function,
$a^i_{\tilde x}(x)$, for a point $\tilde x\in X^i$ is defined on
$[0,1]$ satisfying
\begin{equation*}
  a^i_{\tilde x}(x)=
  \begin{dcases}
    1, & x = \tilde x; \\
    0, & x\in X^i, x\neq \tilde x.
  \end{dcases}
\end{equation*}
For $\dltx^i_j\in \Delta X^i\subseteq X^j$, we define a simplified notation
\begin{equation*}
  a^i_j(x)=a^i_{\dltx^i_j}(x).
\end{equation*}
Figure~\ref{fig_base} shows a few basis functions for the various sparse grids.
\begin{figure}
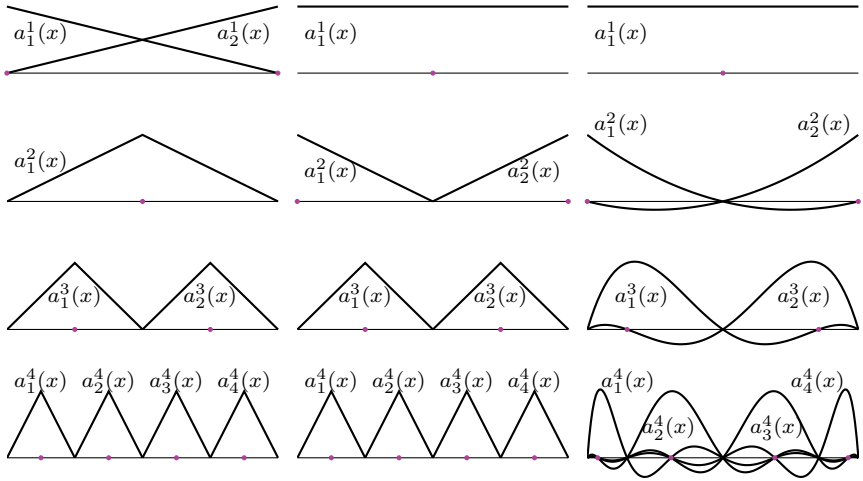

  \centering
%
   \caption{Basis functions for the classic, modified, and CGL sparse
    grids}\label{fig_base}
\end{figure}

The classic sparse grid uses piecewise linear basis functions
\begin{subequations}\label{base_classic}
  \begin{align}
    a^1_1(x)&=x,\\
    a^1_2(x)&=1-x,\\
    a^i_j(x)&=
    \begin{dcases}
      1-\abs{2^{i-1}x-(2j-1)}, & x\in \left[\frac{2(j-1)}{2^{i-1}},
                                 \frac{2j}{2^{i-1}}\right]; \\
      0,                       & \mbox{otherwise};
    \end{dcases}\label{base_classic_gen}
  \end{align}
\end{subequations}
for $i \geq 2$ and $1\leq j\leq 2^{i-2}$.
For the modified sparse grid, $a^i_j(x)$ is the same as in
identity~\eqref{base_classic_gen} if $i\geq 3$ and otherwise
\begin{subequations}\label{base_modified}
  \begin{align}
    a^1_1(x)&=1,\\
    a^2_1(x)&=
      \begin{dcases}
        1-2x, & x\in \left[0, \frac{1}{2}\right];\\
           0, & x\in \left(\frac{1}{2}, 1\right];
      \end{dcases}\\
    a^2_2(x)&=
      \begin{dcases}
           0, & x\in \left[0, \frac{1}{2}\right];\\
        2x-1, & x\in \left(\frac{1}{2}, 1\right].
      \end{dcases}
  \end{align}
\end{subequations}
For the CGL sparse grid, the Lagrange polynomials form the basis functions
\begin{equation*}
  a^i_j(x)=
    \prod_{\substack{x^i\in X^i \\ x^i \ne x^i_j}}\frac{x-x^i}{x^i_j-x^i},
\end{equation*}
for $i > 1$ and $1\leq j\leq 2^{i-2}$.

In the following, the inequality
$\bj \leq \dltN_{\bi}$
implies
\begin{align*}
        j_1\leq \dltN_{i_1},
  \quad j_2\leq \dltN_{i_2},
  \quad \ldots,
  \quad j_d\leq \dltN_{i_d}.
\end{align*}
The inequality $\bk \leq N_{\bi}$ is similarly defined.
The interpolation on a sparse grid does not need every basis function. In fact,
for each $i > 1$, an interpolation function uses only those $a^i_{\tilde x}$
for which $\tilde x\in \Delta X^i$. Let $\mathcal{I}^q(f)$ be the interpolation of $f$ at
grid points of $\Gsp^q$. It is defined recursively on ${[0, 1]}^d$ as
\begin{subequations}\label{interpo}
  \begin{align}
    \mathcal{I}^{d-1}(f)&=0,\\
    \mathcal{I}^q(f)    &= \mathcal{I}^{q-1}+\Delta \mathcal{I}^q (f), & q\geq d,\\
    \Delta \mathcal{I}^q(f) &= \sum_{\abs{\bi}=q}\sum_{1\leq \bj\leq \dltN_{\bi}}
                    w^{\bi}_{\bj}a^{i_1}_{j_1}\otimes\cdots \otimes
                    a^{i_d}_{j_d}, & q\geq d,\\
                    w^{\bi}_{\bj} &=f(x^{\bi}_{\bj})- \mathcal{I}^{q-1}(f)(x^{\bi}_{\bj}),
  \end{align}
where
  \begin{equation*}
    a^{i_1}_{j_1}\otimes\cdots \otimes a^{i_d}_{j_d}(x_1,\ldots,x_d)=
        a^{i_1}_{j_1}(x_1)\cdots  a^{i_d}_{j_d}(x_d).
  \end{equation*}
\end{subequations}
The weights, $w^{\bi}_{\bj}$, are scalars called \emph{hierarchical surpluses}.

An alternative formulation uses basis functions $a^i_{\tilde x}(x)$ for
$\tilde x \in X^i$.
The corresponding basis functions are denoted by
\begin{equation*}
  u^i_k(x)=a^i_{x^i_k}(x),
\end{equation*}
for $x^i_k\in  X^i$.  Note that the subindex in this notation, $k$, represents
the index in $X^i$, not in $\Delta X^i$.  Again, we use piecewise linear basis
functions for the classic and modified sparse grids as in the
definitions~\eqref{base_classic} and~\eqref{base_modified}.  Similarly, we use
Lagrange polynomials as the basis functions for the CGL sparse grid.
For the $X^i$ defined in the equations~\eqref{classicalX}--\eqref{CGL}, we have
the relation
\begin{equation*}
  a^i_j(x)=u^i_{2j-1}(x),
\end{equation*}
for $1\le j \le \dltN_{i_d}$.
The interpolations in $[0,1]$ and ${[0,1]}^d$ are
\begin{align*}
  {\mathcal{U}}^{i} (f) &= \sum_{k=1}^{N_i} f(x^i_k) u^i_k,\\
    {\mathcal{U}}^{i_1}\otimes {\mathcal{U}}^{i_2}\otimes \cdots \otimes
  {\mathcal{U}}^{i_d} (f) &= \sum_{1\leq \bk \leq N_{\bi}}
    f(x^{\bi}_{\bk})u^{i_1}_{k_1}\otimes u^{i_2}_{k_2}\otimes \cdots \otimes
    u^{i_d}_{k_d}.
\end{align*}
Let $\mathcal{I}^q(f)$ be the interpolation on sparse grids. Following the
formulation in \citet{barthelmann}, \citet{delvos}, and \citet{wasilkowski},
the interpolation on the sparse grid, $\mathcal{I}^q(f)$, in terms of the new
basis is
\begin{equation}\label{interpolation2}
  \begin{split}
    \mathcal{I}^q(f)&=\sum_{q-d+1\leq \abs{\bi}\leq q} {(-1)}^{q-\abs{\bi}}
                      \begin{pmatrix} d-1\\ q-\abs{\bi}\end{pmatrix}
                      {\mathcal{U}}^{i_1}\otimes {\mathcal{U}}^{i_2}\otimes
                      \cdots \otimes {\mathcal{U}}^{i_d} (f)\\
                    &=\sum_{q-d+1\leq \abs{\bi}\leq q}
                      \sum_{1\leq \bk \leq N_{\bi}}
                      {(-1)}^{q-\abs{\bi}}
                      \begin{pmatrix} d-1\\ q-\abs{\bi}\end{pmatrix}
                      f(x^{\bi}_{\bk}) u^{i_1}_{k_1}\otimes
                      u^{i_2}_{k_2}\otimes \cdots
                      \otimes u^{i_d}_{k_d}.
  \end{split}
\end{equation}

Before the next section on error analysis, the algorithm is summarized as
follows. The off-line computation has three steps.
\begin{description}[leftmargin=8em,labelindent=3em,style=nextline,itemsep=1ex]
  \item[{\bf Step I}] Generating a sparse grid, $\Gsp^q$, and its basis functions
          $\Set*{a^{i}_{j}}_{i=1,j=1}^{q-d+1,\dltN_i}$.
  \item[{\bf Step II}]  Solve the two-point BVP defined in
    equation~\eqref{eq:hamchar} at each grid point.
  \item[{\bf Step III}] Generating the hierarchical surpluses, $\{ w^i_j\}$,
    using \eqref{interpo} or other equivalent formulae.
\end{description}
The output of the process consists of the values of $\vec{\lambda}(t)$ and
$V(t,x)$ at grid points in $\Gsp^q$. For on-line optimal feedback control, the
value of these functions at an arbitrary $x$ is approximated using interpolation
on the sparse grid. Then the control input \eqref{optuv} is computed for
real-time control and operation. For a feedback using model predictive control,
like the examples in Section~\ref{sec_examples}, the feedback requires the value
of $\vec{\lambda}(t)$ at $t=0$ in each time interval. It is not necessary to
solve the problem for $t\neq 0$. In the case of interpolation at $t\geq 0$,
which is not addressed in this paper, the time variable can be included in the
sparse grid as another dimension.

\section{Error Analysis}
\label{sec_errorbounds}
Let $V^{\bi}_{\bj}$ represent the value of $V(t,x)$ evaluated at
$x=x^{\bi}_{\bj}$, i.e.,
\begin{equation*}
  V^{\bi}_{\bj} = V(t,x^{\bi}_{\bj}).
\end{equation*}
A causality free algorithm, such as the numerical solution of the characteristic
equations \eqref{eq:hamchar}--\eqref{optuv}, approximates
$V^{\bi}_{\bj}$ with an error
\begin{equation}
\label{eq_ebvp}
\bar V^{\bi}_{\bj} = V^{\bi}_{\bj}+\epsilon^{\bi}_{\bj}.
\end{equation}
At an arbitrary point $x$, the approximation based on sparse grid interpolation
is
\begin{equation*}
  \mathcal{I}^q(\bar V)(x)=\mathcal{I}^q(V)(x)+\ebvp,
\end{equation*}
where $\ebvp$ is the error due to $\epsilon^{\bi}_{\bj}$, the numerical error of
the solution of the BVP~\eqref{eq:hamchar}--\eqref{optuv}. More specifically,
\begin{equation}\label{ebvp_explicit}
\ebvp=\sum_{q-d+1\leq \abs{\bi}\leq q} \sum_{1\leq \bk \leq N_{\bi}}
{(-1)}^{q-\abs{\bi}}\begin{pmatrix} d-1\\ q-\abs{\bi}\end{pmatrix}
\epsilon^{\bi}_{\bk}u^{i_1}_{k_1}\otimes u^{i_2}_{k_2}\otimes \cdots
\otimes u^{i_d}_{k_d}.
\end{equation}
Relative to the true value, the interpolation process has an error
\begin{equation*}
  \einterp=\mathcal{I}^q(V)(x)-V(t,x).
\end{equation*}
Therefore,
\begin{equation}\label{eq_error}
\mathcal{I}^q(\bar V)(x)=V(t,x)+\einterp+\ebvp.
\end{equation}

\subsection{Error upper bounds}
For a survey of sparse grids and error estimation, the reader is referred to
\citet{bungartz} and \citet{garcke}. In the case of a classic sparse grid,
applying a piecewise linear interpolation to functions with continuous second
order partial derivatives, $\einterp$ is at the order of
\begin{equation}\label{eq_rate1}
  \norm{\einterp}=O\left(\frac{{(\log N)}^{d-1}}{N^2}\right)
\end{equation}
which holds both for the $L^2$- and $L^\infty$-norm.

The estimate in identity~\eqref{eq_rate1} holds for both $\norm{\einterp}_{L^2}$
and $\norm{\einterp}_{L^\infty}$. In the case of a CGL sparse grid, an error
upper bound of $\einterp$ is proved by \citet{barthelmann}. Suppose a function
$f$ has $k$th order continuous partial derivatives. Define the norm
\[
  \norm{f}_{W^{k,\infty}}=\max\Set*{\norm{
      \frac{\partial^{\abs{\bi}}}{\partial x_1^{i_1}\cdots\partial x_d^{i_d}}
      f}_{L^\infty} \given 1\leq i_1,\ldots,i_d\leq k},
\]
then the interpolation of $f$ on a CGL sparse grid satisfies
\[
  \norm{\einterp}_{W^{k,\infty}} =
  O\left(\frac{{(\log M)}^{(k+2)(d-1)+1}}{M^k}\right),
\]
where $M$ is the number of sparse grid points.

In this section, we prove an upper bound for $\ebvp$. Its value is small if the
dimension, $d$, is one or two. However, the error becomes larger when solving
PDEs with a higher dimension. In our examples of $d=6$ and $\ebvp$ is not
negligible.
Let's define
\[\Lambda_i=\max_x \sum_{k=1}^{N_i} \abs{u^i_k(x)}\]
for $i\ge1$.  For polynomial interpolations, this number is the Lebesgue
constant. For any integer $l\geq d$, we define
\[S_l=\sum_{\abs{\bi}=l}\Lambda_{i_1}\Lambda_{i_2}\cdots\Lambda_{i_d}.\]
\begin{theorem}\label{thm1}
(i) Suppose $\epsilon > 0$ is an upper bound of the numerical error at each
grid point, i.e., $\abs{\epsilon^{\bi}_{\bj}} $ in equation~\eqref{eq_ebvp} are
smaller than $\epsilon$. Then
\begin{equation}\label{eq_ebvpbound}
  \norm{\ebvp}_{L^\infty} < \epsilon \sum_{l=q-d+1}^q
\begin{pmatrix} d-1\\ q-l\end{pmatrix} S_{l}
\end{equation}
(ii) Suppose $\tilde \Lambda_q>0$ is a constant such that
\[\Lambda_i\leq \tilde \Lambda_q \qquad (1\leq i\leq q-d+1).\]
Then
\begin{equation}\label{eq_order1}
  \frac{\norm{\ebvp}_{L^\infty}}{\epsilon}=
    O\left({(\log N)}^{d-1} \tilde \Lambda_q^d\right)
\end{equation}
where $N+1=2^{q-d}+1$ is the number of grid points in each dimension.
\end{theorem}
\begin{proof}
(i)
It is easy to check that
\begin{align*}
  \sum_{\abs{\bi}=l}\sum_{1\leq \bk \leq N_{\bi}}
  \abs{u^{i_1}_{k_1}\otimes u^{i_2}_{k_2}\otimes \cdots \otimes u^{i_d}_{k_d}}
  &=\sum_{\abs{\bi}=l}\sum_{k_1=1}^{N_1} \cdots \sum_{k_d=1}^{N_d}
  \abs{u^{i_1}_{k_1}\otimes u^{i_2}_{k_2}\otimes \cdots \otimes u^{i_d}_{k_d}}\\
  &=\sum_{\abs{\bi}=l}\left(\sum_{k_1=1}^{N_1} \abs{u^{i_1}_{k_1}}\right)\otimes \cdots
  \otimes \left(\sum_{k_d=1}^{N_d} \abs{u^{i_d}_{k_d}}\right)\\
  &\leq \sum_{\abs{\bi}=l} \Lambda_{i_1}\Lambda_{i_2}\cdots\Lambda_{i_d}
\end{align*}
for all $x\in {[0, 1]}^d$. Therefore, given any integer $l\geq d$ we have
\begin{equation}\label{eq_ineqSq}
\sum_{\abs{\bi}=l}\sum_{1\leq \bk \leq N_{\bi}}
\abs{u^{i_1}_{k_1}\otimes u^{i_2}_{k_2}\otimes \cdots
\otimes u^{i_d}_{k_d} (x)} \leq S_l.
\end{equation}
Then the upper bound \eqref{eq_ebvpbound} is a corollary of
equations~\eqref{interpolation2}, \eqref{ebvp_explicit}, and \eqref{eq_ineqSq}.

(ii)
To find the order of $\norm{\ebvp}/\epsilon$, we first note
\begin{equation*}
\begin{split}
  S_l &=    \sum_{\abs{\bi}=l}\Lambda_{i_1}\Lambda_{i_2}\cdots\Lambda_{i_d}\\
      &\leq \tilde \Lambda_q^d \sum_{\abs{\bi}=l} 1\\
      &=    \begin{pmatrix} l-1\\ l-d\end{pmatrix} \tilde \Lambda_q^d\\
      &=    \frac{(l-1)(l-2)\cdots (l-d)\cdots d}{(l-d)!} \tilde \Lambda_q^d\\
      &\leq {(l-1)}^{d-1} \tilde \Lambda_q^d\\
      &\leq {(q-1)}^{d-1} \tilde \Lambda_q^d.
\end{split}
\end{equation*}
From identity~\eqref{eq_ebvpbound},
\begin{equation*}
\begin{split}
  \frac{\norm{\ebvp}_{L^\infty}}{\epsilon}
  &<    \sum_{l=q-d+1}^q \begin{pmatrix} d-1\\ q-l\end{pmatrix} S_{l}\\
  &\leq {(q-1)}^{d-1}  \tilde \Lambda_q^d \sum_{l=q-d+1}^q
        \begin{pmatrix} d-1\\ q-l\end{pmatrix} \\
  &=    2^{d-1}{(q-1)}^{d-1}  \tilde \Lambda_q^d.
\end{split}
\end{equation*}
Then the relation~\eqref{eq_order1} follows from the fact
\[
q=d+\log_2 N= O(\log N).
\]
\qed{}
\end{proof}

\begin{corollary}\label{corollary}
For the classic or the modified sparse grid and piecewise linear
interpolation,
\begin{equation}\label{eq_order2}
  \frac{\norm{\ebvp}_{L^\infty}}{\epsilon}=O\left({(\log N)}^{d-1} \right).
\end{equation}
For the CGL sparse grid and polynomial interpolation,
\begin{equation}\label{eq_order3}
  \frac{\norm{\ebvp}_{L^\infty}}{\epsilon}=O\left({(\log N)}^{2d-1} \right).
\end{equation}
\end{corollary}
\begin{proof}
The basis functions of the piecewise linear interpolation have the following
property
\[\tilde \Lambda_q =1.\]
The Lebesgue constant for the CGL grid points is bounded by (see
for example \citet{hesthaven})
\begin{equation}\label{lebesgue}
  \tilde \Lambda_q \leq \frac{2}{\pi}\log N+\frac{2}{\pi}\left(\gamma
  +\log\frac{4}{\pi}\right)+\frac{2}{\pi}\log 2
                 =     O\left(\log N\right)
\end{equation}
where $\gamma$ is Euler's constant.
Substitute $\tilde \Lambda_q$ into equation~\eqref{eq_order1} to yield
equations~\eqref{eq_order2} and \eqref{eq_order3}. \qed
\end{proof}

\subsection{Numerically estimate $\ebvp$}
Some numerical examples show that the upper bound in
inequality~\eqref{eq_ebvpbound} tends to be conservative. Using
relations~\eqref{eq_ebvpbound} and~\eqref{lebesgue}, we can find an error upper
bound for the CGL sparse grid $\Gsp^q$. For example, if $q=13$ we have
\begin{equation}\label{eq_upperbound_ad}
\frac{\norm{\ebvp}_{L^\infty}}{\epsilon} < \sum_{l=q-d+1}^q
\begin{pmatrix} d-1\\ q-l\end{pmatrix} S_{l} \leq 3.66 \times 10^4.
\end{equation}
In this estimation, $\ebvp=\epsilon^\bi_\bj$ at each grid point is assumed to be
the maximum value, $\epsilon$. It is a conservative assumption.  As an
alternative, we can assume $\epsilon^\bi_\bj$ is a random variable.
(This approach is commonly used in uncertainty quantification, rightly or
wrongly, to get a handle on the model error.)  We will do this here to get
another estimate of $\ebvp$.

From identity~\eqref{ebvp_explicit}, $\ebvp$ has two special
properties. Firstly, the value of $\ebvp$ is not directly dependent on $f(x)$.
The error is solely based on the type of grid, the interpolation basis
functions, and the distribution of $\epsilon^\bi_\bj$.  Therefore, the
estimation of $\ebvp$ can be done off-line for a given grid. The result is then
applicable to a family of problems. Secondly, $\ebvp$ is a linear function of
$\epsilon^\bi_\bj$. The estimate of $\ebvp$ for a given distribution of
$\epsilon^\bi_\bj$ is applicable to a different distribution of the same type
through a simple rescaling.

Given a sample set of $\epsilon^\bi_\bj$, $\ebvp$ can be computed using
equation~\eqref{ebvp_explicit} at any $x$ in ${[0, 1]}^d$. Finding the
probability distribution of $\epsilon^\bi_\bj$ is not a problem to be addressed
in this paper. As an example, we assume that $\epsilon^\bi_\bj$ are independent
random variables with uniform distribution in $[-\epsilon, \epsilon]$.  Let
$\Set*{\bar\epsilon^\bi_\bj}$ be a sample data set with a uniform distribution
in $[0,1]$. After rescaling, equation~\eqref{ebvp_explicit} implies
\begin{equation*}\label{ebvp_explicit2}
\frac{\ebvp}{\epsilon}=\sum_{q-d+1\leq \abs{\bi}\leq q}
\sum_{1\leq \bk \leq N_{\bi}} {(-1)}^{q-\abs{\bi}}
\begin{pmatrix} d-1\\ q-\abs{\bi}\end{pmatrix}\bar\epsilon^{\bi}_{\bk}u^{i_1}_{k_1}
  \otimes u^{i_2}_{k_2}\otimes \cdots \otimes u^{i_d}_{k_d}.
\end{equation*}
In an numerical example, a set of $\abs{\Gsp^q}=44,689$ random numbers are generated
as the sample value, $\bar \epsilon^\bi_\bj$, in $[0, 1]$. At $N=2,000$ random
points in ${[0,1]}^d$, we found
\begin{equation}\label{eq_66}
\frac{\ebvp}{\epsilon}\leq 66.60.
\end{equation}
The bound in \eqref{eq_66} is much smaller than the upper bound in
the inequality~\eqref{eq_upperbound_ad} derived from
Theorem~\ref{thm1}.
We would like to emphasize that this practical way of estimating
$\ebvp$ is independent of the function to be approximated. The overall
error $\einterp+\ebvp$ can be approximated in a similar way based on
the assumption that $\epsilon$, the error upper bound of BVP solver,
is known. This point is discussed in the next section and exemplified
in Example~I.  A thorough numerical analysis of errors is outside the
scope of this paper. By no means can the examples in this paper lead
to a general conclusion about the error upper bound.  However, the
examples in this paper present a practical approach of using Monte
Carlo simulations to analyze $\ebvp$ and $\einterp+\ebvp$.

\subsection{Numerically estimate $\einterp+\ebvp$}
\label{sec_totalerror}
What we ultimately care about is the total error in an approximate of
$V(t,x)$, which is $\einterp+\ebvp$ in equation~\eqref{eq_error}. For causality
free algorithms, such as solving the BVP~\eqref{eq:hamchar}--\eqref{optuv},
numerical errors do not propagate in space, i.e., the value of $V(t,x)$ can be
computed independently from the approximation error at other points. This
special property of causality free algorithms has an important implication. A
BVP solver with accurate error control can be used to approximate the error of a
numerical solution of the PDE.  There is a sizable literature of numerical BVP
algorithms. Some approaches are able to control the error within a given
tolerance. For the examples in this paper, we use a four-point Lobatto IIIa
formula which can be implemented with a controlled true error~\cite{kierzenka}.
Given a numerical solution on $\Gsp^q$,
\[\bar V^\bi_\bj\approx V(t, x^\bi_\bj).\]
The approximate of $V(t,x)$ at any point $x$ is obtained by interpolation
\[V(t,x)\approx \mathcal{I}^q(\bar V).\]
Meanwhile, a BVP solver with error control can be applied to solve the
equations~\eqref{eq:hamchar}--\eqref{optuv} at the same point. Suppose the
solution is $\tilde V(t,x)$.  The error tolerance is set so that its true error
is much smaller than $\einterp+\ebvp$. To find the distribution of
$\einterp+\ebvp$, a set of sample points is generated in the state space,
organized or random. The numerical error at these points can be approximated
using
\begin{equation*}\label{validate}
  \abs{\einterp+\ebvp}=\abs{\mathcal{I}^q(\bar V)(x)-V(t,x)}\approx
    \abs{\mathcal{I}^q(\bar V)(x)-\tilde V(t,x)}.
\end{equation*}
Although the computation of $\tilde V(t,x)$ can be slow, the entire process is
perfectly parallel because of its nature of being causality free. In the
examples below, 128 CPU cores are used to approximate the error at $1280$ random
points. The error tolerance of $\tilde V(t,x)$ is set at $10^{-7}$ or
$10^{-9}$. The numerical results show that this tolerance is smaller than
$\abs{\einterp+\ebvp}$ by at least three orders of magnitude.

\section{Examples}
\label{sec_examples}
In this section, two examples of optimal attitude control are presented. The
system model represents a rigid body controlled by momentum wheels. The first
example is a system with three pairs of controllable momentum wheels.
The second example is an uncontrollable system with two pairs of momentum
wheels. The uncontrollable case in itself is interesting. The numerical result
is fundamentally different from existing controllers in the literature.

Consider a rigid body system. Let $\Set*{e_1, e_2, e_3}$ be an inertial frame of
orthonormal vectors and let $\Set*{e_1', e_2', e_3'}$ be a body-fixed frame, or
body frame. In this paper, the attitude of a satellite is represented by Euler
angles (see \citet{diebel} for a good introduction to representing attitude)
\[v=\begin{bmatrix} \phi & \theta & \psi\end{bmatrix}^\mathsf{T}\]
in which $\phi$, $\theta$, and $\psi$ are the angles of rotation around $e_1'$,
$e_2'$, and $e_3'$, respectively, in the order of $(3,2,1)$. The angular
velocity is a vector in the body frame,
\[\omega = \begin{bmatrix} \omega_1&\omega_2&\omega_3\end{bmatrix}^\mathsf{T}.\]
The control system using momentum wheels is defined by a set of differential
equations~\cite{crouch}
\begin{subequations}\label{sat_model}
  \begin{align}
          \dot v &= E(v) \omega, \\
    J\dot \omega &= S(\omega)R(v)H+Bu,
  \end{align}
\end{subequations}
where $B\in \Re^{3\times m}$ is a constant matrix in which $m$ is the number of
control variables, $u\in\Re^m$ is the control torque, $J \in \Re^{3\times 3}$ is a
combination of inertia matrices of the rigid body without wheels and the
momentum wheels, $H\in \Re^3$ is the total and constant angular momentum of the
system, and $E(v), S(\omega), R(v) \in \Re^3$ are matrices. Details can be found
in \citet{kangwilcox}.

\subsection{Example~I}
The system has three pairs of control momentum wheels, $m=3$. In
the model~\eqref{sat_model}, the following parameter values are used
\begin{align*}\label{para_1}
  B&=\begin{bmatrix} 1& \frac{1}{20}& \frac{1}{10}\\\frac{1}{15} &1& \frac{1}{10}\\
     \frac{1}{10}& \frac{1}{15}& 1 \end{bmatrix}, &
  J&=\begin{bmatrix} 2 & 0 & 0\\ 0 & 3& 0\\ 0&0&4 \end{bmatrix}, &
  H&=\begin{bmatrix} 1 \\ 1 \\ 1\end{bmatrix}.
\end{align*}
The optimal control is
\begin{equation*}\label{sat_cost1}
\argmin_u \int_t^{T} L(v,\omega,u)\;ds+W_4\norm{v(T)}^2+W_5\norm{\omega(T)}^2
\end{equation*}
where
\begin{gather*}\label{para_1_2}
L(v,\omega,u)=\frac{W_1}{2} \norm{v}^2+\frac{W_2}{2}
              \norm{\omega}^2+\frac{W_3}{2} \norm{u}^2, \\
W_1=1,\quad
W_2=1,\quad
W_3=1/2,\quad
W_4=1,\quad
W_5=1,\quad
t=0,\quad
T=20.
\end{gather*}

Since $L$ is a convex function it can be proved that a unique solution exists in
a neighborhood of the target state.
The solution $V(t,v,\omega)$ is computed at $t=0$ for  initial states $v(0)$ and
$\omega(0)$ in two domains, $D_1$ and $D_2$, of different size,
\begin{align*}
  D_1 &= \Set*{v,\omega\in\Re^3\given
    -\frac{\pi}{6} \leq \phi, \theta, \psi \leq \frac{\pi}{6} \text{ and }
    -\frac{\pi}{8} \leq \omega_1, \omega_2, \omega_3 \leq \frac{\pi}{8}}, \\
  D_2 &= \Set*{v,\omega\in\Re^3\given
    -\frac{\pi}{3} \leq \phi, \theta, \psi \leq \frac{\pi}{3} \text{ and }
    -\frac{\pi}{4} \leq \omega_1, \omega_2, \omega_3 \leq \frac{\pi}{4}}.
\end{align*}
The computation is based on the CGL sparse grid with $q=13$. The number of
grid points for each dimension is $2^{q-6}+1=129$. The total number of grid points
in the $6$-D domain is \[\abs{\Gsp^q}=44,689,\]
which is small in comparison with the size of a dense grid,
\[\abs{\Gdn^q}=129^6>4.6\times 10^{12}.\]
In the computation, the two-pont BVP~\eqref{eq:hamchar} is solved at each
grid point in $\Gsp^q$ using a four-stage Lobatto IIIa method~\cite{kierzenka}.
The error tolerance is $10^{-12}$. The computation is carried out in Hamming, a
parallel computer of Naval Postgraduate School. Although as many as
$\abs{\Gsp^q}$
can be used, we limit the computation to $512$ CPU cores. To check the accuracy
of the overall solution, the upper bound of  $\einterp+\ebvp$ is numerically
approximated using the method in Section~\ref{sec_totalerror}. More
specifically, $1280$ points are randomly generated in $D_1$ and $D_2$. At each
sample point,  the value of $\bar V(0,v,\omega)$ is computed using
interpolation. The true value at the same point is approximated by
$\tilde V(0,v,\omega)$, which is computed by applying the BVP solver to
the equation~\eqref{eq:hamchar} with an error tolerance $10^{-9}$ in $D_1$ and
$10^{-7}$ in $D_2$. This tolerance is much smaller than $\einterp+\ebvp$ so that
the error of the BVP solver can be ignored. The difference,
$\bar V(0,v,\omega)-\tilde V(0,v,\omega)$, is an approximate of $\einterp+\ebvp$
at the sample point. In $D_1$, the mean absolute error (MAE) is $4.9\times
10^{-7}$. The MAE for the relative error is $4.0\times 10^{-7}$. In the larger
domain $D_2$, the MAE equals $3.6\times 10^{-3}$ and the relative error is
$7.3\times 10^{-4}$. The results are summarized in Table~\ref{table1}.
\begin{table}
  \centering
  \begin{tabular}{@{}cccccc@{}} \toprule
    Domain & $q$ & $\abs{\Gdn^q}$ & $\abs{\Gsp^q}$ & MAE & Relative MAE \\ \midrule
    $D_1$&$13$&$>10^{12}$&$44,698$& $4.9\times 10^{-7}$&$4.0\times 10^{-7}$\\
    $D_2$&$13$&$>10^{12}$&$44,698$& $3.6\times 10^{-3}$&$7.3\times 10^{-4}$\\
    \bottomrule
  \end{tabular}
  \caption{Summary of results for Example~I}\label{table1}
\end{table}
Figure~\ref{fig_trajectory} shows a typical trajectory in which $(v,\omega)$
converges to zero, i.e., the rigid body is stabilized.
\begin{figure}
  \centering
\begin{tikzpicture}
\matrix {
\begin{axis}[%
width=\figurewidth,
height=\figureheight,
scale only axis,
xmin=0,
xmax=20,
xlabel={time},
ymin=-1,
ymax=1.5,
ylabel={$v$}
]
\addplot [color=figcolor0,solid,forget plot]
  table[row sep=crcr]{%
0	1.0471975511966\\
0.296296296296296	0.893051417128452\\
0.777777777777778	0.72138855051856\\
1.58333333333333	0.553474423921305\\
2.54166666666667	0.437749242248953\\
3.83333333333333	0.314729553811096\\
5.77777777777778	0.173236121359662\\
9.33333333333333	0.0513642532070097\\
15.5	0.00552491018180756\\
19.7777777777778	0.00122869486372032\\
};\label{plot_solid}
\addplot [color=figcolor1,dashed,forget plot]
  table[row sep=crcr]{%
0	-0.740480489693061\\
0.296296296296296	-0.633598000840963\\
0.777777777777778	-0.506994517231086\\
1.58333333333333	-0.377645596603826\\
2.54166666666667	-0.265620459254615\\
3.83333333333333	-0.149372662744813\\
5.77777777777778	-0.0574314687166416\\
9.33333333333333	-0.0121627124436098\\
15.5	-0.00143626772781583\\
19.7777777777778	-0.000258044880405572\\
};\label{plot_dashed}
\addplot [color=figcolor2,dotted,forget plot]
  table[row sep=crcr]{%
0	1.0471975511966\\
0.296296296296296	1.22785930267678\\
0.777777777777778	1.28531997033246\\
1.58333333333333	1.14363154435223\\
2.54166666666667	0.904085728721756\\
3.83333333333333	0.630023871822248\\
5.77777777777778	0.355949036634005\\
9.33333333333333	0.122399820657945\\
15.5	0.0195166278028748\\
19.7777777777778	0.00739514567784632\\
};\label{plot_dotted}
\end{axis}
\\
\begin{axis}[%
width=\figurewidth,
height=\figureheight,
scale only axis,
xmin=0,
xmax=20,
xlabel={time},
ymin=-0.5,
ymax=1,
ylabel={$\omega$}
]
\addplot [color=figcolor0,solid,forget plot]
  table[row sep=crcr]{%
0	0\\
0.296296296296296	-0.228103501098306\\
0.777777777777778	-0.312335084275991\\
1.58333333333333	-0.239914038735321\\
2.54166666666667	-0.166871918470562\\
3.83333333333333	-0.114797151792388\\
5.77777777777778	-0.0627926433970611\\
9.33333333333333	-0.01847732115012\\
15.5	-0.00205894237205846\\
19.7777777777778	-0.000407118679681448\\
};
\addplot [color=figcolor1,dashed,forget plot]
  table[row sep=crcr]{%
0	0.785398163397448\\
0.296296296296296	0.427981047950935\\
0.777777777777778	0.126958161555234\\
1.58333333333333	-0.00774144053350878\\
2.54166666666667	-0.00231587433465037\\
3.83333333333333	0.0126674191422093\\
5.77777777777778	0.00920152027362501\\
9.33333333333333	0.00282760986488173\\
15.5	0.000439227685930856\\
19.7777777777778	0.000243268359716204\\
};
\addplot [color=figcolor2,dotted,forget plot]
  table[row sep=crcr]{%
0	0\\
0.296296296296296	-0.0703882807026952\\
0.777777777777778	-0.173868315819247\\
1.58333333333333	-0.259803312696575\\
2.54166666666667	-0.256258525532338\\
3.83333333333333	-0.192756337953466\\
5.77777777777778	-0.108935334558142\\
9.33333333333333	-0.0371197085602627\\
15.5	-0.00552270707236853\\
19.7777777777778	-0.00108442273294336\\
};
\end{axis}
\\
\begin{axis}[%
width=\figurewidth,
height=\figureheight,
scale only axis,
xmin=0,
xmax=20,
xlabel={time},
ymin=-6,
ymax=2,
ylabel={$u$}
]
\addplot [color=figcolor0,solid,forget plot]
  table[row sep=crcr]{%
0	-1.70648891800618\\
0.296296296296296	-0.465925462941883\\
0.777777777777778	0.185053387967175\\
1.58333333333333	0.180161564073538\\
2.54166666666667	0.131688281526989\\
3.83333333333333	0.158621918814463\\
5.77777777777778	0.128858879543878\\
9.33333333333333	0.0491232014591082\\
15.5	0.00720287280822893\\
19.7777777777778	0.0010563137095275\\
};
\addplot [color=figcolor1,dashed,forget plot]
  table[row sep=crcr]{%
0	-4.28397169820891\\
0.296296296296296	-2.63179227623886\\
0.777777777777778	-1.14825307545261\\
1.58333333333333	-0.371346407488113\\
2.54166666666667	-0.275731036031303\\
3.83333333333333	-0.233095989713568\\
5.77777777777778	-0.119867517174896\\
9.33333333333333	-0.0325496369427432\\
15.5	-0.00489982031000198\\
19.7777777777778	-0.000437255230382117\\
};
\addplot [color=figcolor2,dotted,forget plot]
  table[row sep=crcr]{%
0	-1.74594369835198\\
0.296296296296296	-1.36168314420169\\
0.777777777777778	-0.748323819694147\\
1.58333333333333	-0.108215975133827\\
2.54166666666667	0.13677120505612\\
3.83333333333333	0.137650364592568\\
5.77777777777778	0.0681491597535764\\
9.33333333333333	0.0220630921193882\\
15.5	0.0043403481823754\\
19.7777777777778	0.000883451459792729\\
};
\end{axis}
\\
};
\end{tikzpicture}%
   \caption{Trajectories (Solid~\ref{plot_solid}: $\phi$, $\omega_1$, $u_1$;
    dashed~\ref{plot_dashed}: $\theta$, $\omega_2$, $u_2$;
    dotted~\ref{plot_dotted}: $\psi$, $\omega_3$, $u_3$)}\label{fig_trajectory}
\end{figure}
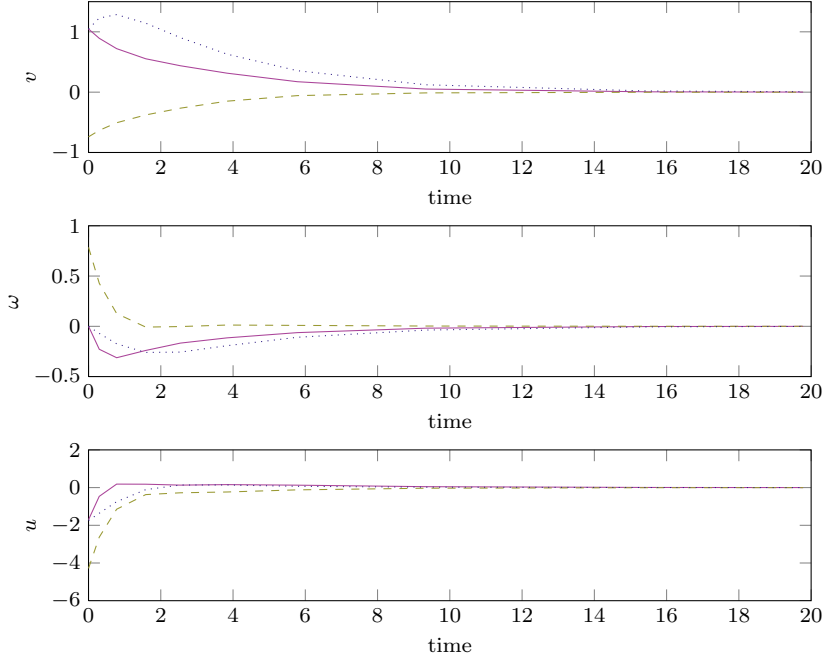
The accuracy of the solution, especially in $D_2$, is not as high as PDE solvers
in some $2$-D or $3$-D cases. This is not surprising because we sacrifice the
accuracy in exchange for a sparse grid that is tractable in $6$-D. For a fixed
dimension $d$, the accuracy of the solution at an arbitrary point depends on $q$
and the size of the domain. For a fixed $q$, the accuracy is increase when the
size of the domain is decreased. Shown in Table~\ref{table1}, decreasing the
linear dimension of $D_2$ by $50\%$ increases the accuracy by three orders of
magnitude.

\subsection{Example~II}
In the following, we consider a rigid body controlled by two pairs of momentum
wheels. Although satellite systems are quipped with at least three pairs of
control momentum wheels, malfunction may occur in some wheels. It is proved by
\citet{crouch} that this system is uncontrollable. How to stabilize the satellite
around a desired attitude is a challenging problem. Related work can be found in
\citet{tsiotras,gui,kim,krishnan,terui,horri2} and references therein. Different
from exiting results, we do no assume zero angular momentum. In addition, the
controller based on the solution of the associated HJB equation is smooth. The
optimal control is able to smoothly stabilize the rigid body at an attitude that
is closest to the desired orientation. The result in this paper is different
from those in \citet{kangwilcox} where the stabilization does not guarantee an
optimal attitude. In the present paper, we optimize a cost function that
automatically stabilizes the system at an optimal attitude. In addition, the
result is integrated with a model predictive control (MPC) to achieve feedback
stabilization in the presence of noise.

In this section, the following values are assigned to the parameters
\begin{align*}\label{para_1_3}
  B&=\begin{bmatrix} 1& \frac{1}{10}\\0 &1\\ \frac{1}{12}& 0\end{bmatrix}, &
  J&=\begin{bmatrix} 2 & 0 & 0\\ 0 & 3& 0\\ 0&0&4 \end{bmatrix}, &
  H&=\begin{bmatrix} 12 \\ 12 \\ 6 \end{bmatrix}.
\end{align*}
The optimal control is
\begin{equation}\label{sat_cost2}
\argmin_u \int_t^{T} L(v,\omega,u)\;ds \\
\end{equation}
where
\begin{gather*}\label{para_2}
L(v,\omega,u)=\frac{W_1}{2} \norm{v-v_e(v,\omega)}^2+\frac{W_2}{2}
\norm{\omega}^2+\frac{W_3}{2} \norm{u}^2, \\
W_1=1,   \quad
W_2=2,   \quad
W_3=0.5, \quad
t_0=0,   \quad
T=30.
\end{gather*}
The function $v_e(v,\omega)$ represents the optimal attitude reachable from
$(v,\omega)$. It is a known fact that this system is uncontrollable. The desired
attitude, in our example $v=0$, may not be reachable. In the case of $H=0$,
nonsmooth controllers can be derived to stabilize the system
\cite{tsiotras,gui,kim,krishnan,terui,horri2}. In the case of $H\neq 0$, a
manifold of reachable states $(v,\omega)$ satisfies~\cite{crouch}
\begin{gather*}
  C^\mathsf{T}(J\omega -R(v)H)= \text{constant},\\
  \text{for some } C\in \Re^3 \text{ such that } C^\mathsf{T} B=0.
\end{gather*}
The attitude $v_e(v,\omega)$ is a target attitude in this reachable manifold. A
satellite system may have to meet multiple requirements of orientation, such as
pointing sensors to the desired direction and at the same time keeping its solar
panel facing the Sun. But the desired attitude, for instance $v=0$, may not lie
on the manifold of reachable states. We define the following optimal
$v_e(v,\omega)$ as the target state for stabilization. For this purpose, we
minimize the Frobenius distance between $R(v)$ and $I=R(0)$. Because both
matrices are orthogonal, it is equivalent to
\begin{subequations}\label{optim_attitude}
  \begin{align}
    v_e(v,\omega)&=\argmax_{\tilde v} {\rm tr} (R(\tilde v)) \\
                 &\text{subject to }
                  -C^\mathsf{T}R(\tilde v)H=C^\mathsf{T}(J\omega -R(v)H).
  \end{align}
\end{subequations}
It can be proved that $v_e(v(t),\omega(t))$, without noise and uncertainty, is a
constant along any controlled trajectory. Therefore, it can be treated as a
constant in the derivation of the BVP~\eqref{eq:hamchar}. The solution of the
maximization problem~\eqref{optim_attitude} can be found by algorithms of
numerical nonlinear programming.

The HJB equation is solved at $t=0$ on a sparse grid in the domain
\begin{equation*}
  D = \Set*{v,\omega\in\Re^3\given
        -\frac{\pi}{6} \leq \phi, \theta, \psi \leq \frac{\pi}{6}
        \text{ and }
        -\frac{\pi}{8} \leq \omega_1, \omega_2, \omega_3 \leq \frac{\pi}{8}}.
\end{equation*}
The CGL sparse grid of $q=13$ is used. Similar to the approach in Example~I, we
solve the HJB equation at the grid points using $512$ CPU cores in parallel. Then
the accuracy is checked at a random of $1280$ point in $D$.  The MAE of the
relative error is $8.5\times 10^{-3}$.  A typical trajectory is shown in
Figure~\ref{fig_traj}.
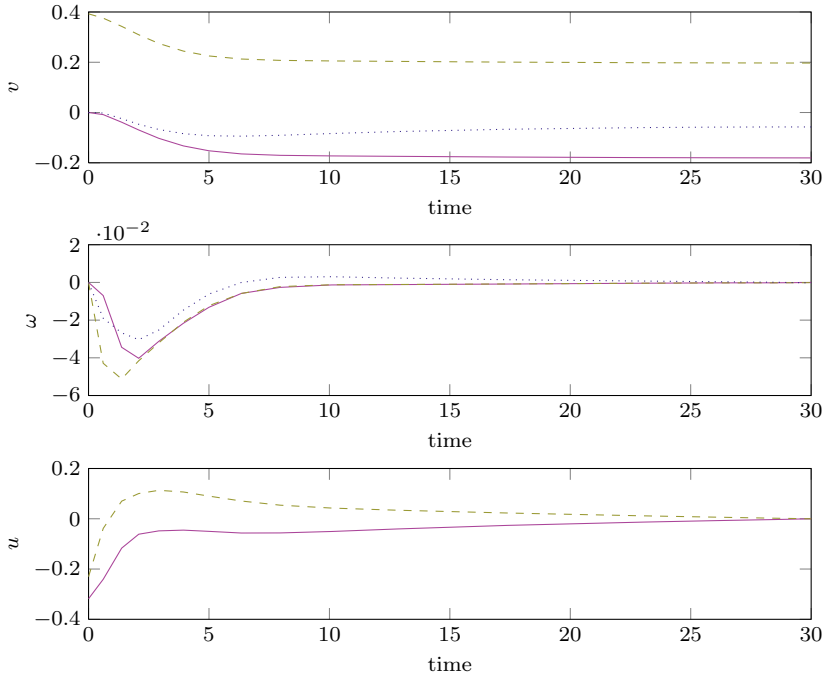
\begin{figure}
  \centering
\begin{tikzpicture}
\matrix{
\begin{axis}[%
width=\figurewidth,
height=\figureheight,
scale only axis,
xmin=0,
xmax=30,
xlabel={time},
ymin=-0.2,
ymax=0.4,
ylabel={$v$}
]
\addplot [color=figcolor0,solid,forget plot]
  table[row sep=crcr]{%
0	0\\
0.609375	-0.00833175921791912\\
1.375	-0.0385561876159122\\
2.08333333333333	-0.0694767193245362\\
2.91666666666667	-0.10264579667894\\
3.9375	-0.132838899755674\\
5	-0.152218033835477\\
6.34375	-0.164665054546017\\
7.95833333333333	-0.17031303094758\\
10.09375	-0.172605539164885\\
12.8125	-0.174289037562987\\
17.3958333333333	-0.177121752323473\\
23.046875	-0.179392208677919\\
26.1111111111111	-0.180028504419029\\
28.4259259259259	-0.180359099637826\\
30	-0.180552777099657\\
};
\addplot [color=figcolor1,dashed,forget plot]
  table[row sep=crcr]{%
0	0.392699081698724\\
0.609375	0.375410639913423\\
1.375	0.342816469923428\\
2.08333333333333	0.30937040826903\\
2.91666666666667	0.274814261347161\\
3.9375	0.244396191265451\\
5	0.225048656096623\\
6.34375	0.212832144289528\\
7.95833333333333	0.207451487727789\\
10.09375	0.205244372458324\\
12.8125	0.203440283947918\\
17.3958333333333	0.200500860050175\\
23.046875	0.198209980308088\\
26.1111111111111	0.197524457027088\\
28.4259259259259	0.197116275053862\\
30	0.196773574457021\\
};
\addplot [color=figcolor2,dotted,forget plot]
  table[row sep=crcr]{%
0	0\\
0.609375	-0.00206097547290578\\
1.375	-0.0243447928490246\\
2.08333333333333	-0.0467496929848373\\
2.91666666666667	-0.0675202344178847\\
3.9375	-0.0840944321659692\\
5	-0.0922578774009528\\
6.34375	-0.0941028327769731\\
7.95833333333333	-0.0905978719103191\\
10.09375	-0.0835495733072006\\
12.8125	-0.0756960847943406\\
17.3958333333333	-0.0665050036988447\\
23.046875	-0.0598825386950542\\
26.1111111111111	-0.0580134805086106\\
28.4259259259259	-0.0573954100610571\\
30	-0.0574016235994851\\
};
\end{axis}
\\
\begin{axis}[%
width=\figurewidth,
height=\figureheight,
scale only axis,
xmin=0,
xmax=30,
xlabel={time},
ymin=-0.06,
ymax=0.02,
ylabel={$\omega$}
]
\addplot [color=figcolor0,solid,forget plot]
  table[row sep=crcr]{%
0	0\\
0.609375	-0.00696798320322616\\
1.375	-0.034402141927204\\
2.08333333333333	-0.0402448880065962\\
2.91666666666667	-0.0312661375964471\\
3.9375	-0.0216742853391357\\
5	-0.0131928078147518\\
6.34375	-0.00584139114551901\\
7.95833333333333	-0.00262641287434505\\
10.09375	-0.00133286148178633\\
12.8125	-0.0011520478084235\\
17.3958333333333	-0.000854878804381666\\
23.046875	-0.000424217941583042\\
26.1111111111111	-0.000237204165733264\\
28.4259259259259	-0.000209570763808833\\
30	-3.7122071996669e-05\\
};
\addplot [color=figcolor1,dashed,forget plot]
  table[row sep=crcr]{%
0	0\\
0.609375	-0.0428918499964976\\
1.375	-0.0512378979898652\\
2.08333333333333	-0.0415787246605317\\
2.91666666666667	-0.0320756139370365\\
3.9375	-0.0211182762167648\\
5	-0.0123355293338229\\
6.34375	-0.00586357782646436\\
7.95833333333333	-0.00222603213983805\\
10.09375	-0.00122597932448041\\
12.8125	-0.00109938454141474\\
17.3958333333333	-0.000811195014414781\\
23.046875	-0.000410582321146871\\
26.1111111111111	-0.000246463930477075\\
28.4259259259259	-0.000207409820201954\\
30	-0.000112462963939527\\
};
\addplot [color=figcolor2,dotted,forget plot]
  table[row sep=crcr]{%
0	0\\
0.609375	-0.0189674190635891\\
1.375	-0.0267360252205321\\
2.08333333333333	-0.0303840501564858\\
2.91666666666667	-0.0252552425310248\\
3.9375	-0.0147056145499386\\
5	-0.00617347565180352\\
6.34375	-3.63588514964474e-05\\
7.95833333333333	0.00267514311656748\\
10.09375	0.00302893931385838\\
12.8125	0.0023044195217443\\
17.3958333333333	0.00142635842395734\\
23.046875	0.000719846356591077\\
26.1111111111111	0.000366176871018977\\
28.4259259259259	0.000100157713040182\\
30	-0.000254035703661541\\
};
\end{axis}
\\
\begin{axis}[%
width=\figurewidth,
height=\figureheight,
scale only axis,
xmin=0,
xmax=30,
xlabel={time},
ymin=-0.4,
ymax=0.2,
ylabel={$u$}
]
\addplot [color=figcolor0,solid,forget plot]
  table[row sep=crcr]{%
0	-0.318140535468118\\
0.609375	-0.240706586200661\\
1.375	-0.116817110292916\\
2.08333333333333	-0.0614441800886596\\
2.91666666666667	-0.048153754521038\\
3.9375	-0.045200261602928\\
5	-0.0501002164319836\\
6.34375	-0.0567187468126443\\
7.95833333333333	-0.0563156081168376\\
10.09375	-0.0505845131469702\\
12.8125	-0.0405542180473966\\
17.3958333333333	-0.026321975598537\\
23.046875	-0.0131945392829119\\
26.1111111111111	-0.00724177589519132\\
28.4259259259259	-0.00318341409021684\\
30	0\\
};
\addplot [color=figcolor1,dashed,forget plot]
  table[row sep=crcr]{%
0	-0.230242244506262\\
0.609375	-0.0382506989950015\\
1.375	0.0700977573720523\\
2.08333333333333	0.100386914207199\\
2.91666666666667	0.113168397820585\\
3.9375	0.106553530123427\\
5	0.0901126471283591\\
6.34375	0.070638190813047\\
7.95833333333333	0.0540296098887216\\
10.09375	0.042556000692799\\
12.8125	0.0339903825062828\\
17.3958333333333	0.0228985509136091\\
23.046875	0.0116451903736274\\
26.1111111111111	0.00627673994512388\\
28.4259259259259	0.00229525645661863\\
30	0\\
};
\end{axis}
\\
};
\end{tikzpicture}%
   \caption{Trajectories (Solid~\ref{plot_solid}: $\phi$, $\omega_1$, $u_1$;
    dashed~\ref{plot_dashed}: $\theta$, $\omega_2$, $u_2$;
    dotted~\ref{plot_dotted}: $\psi$, $\omega_3$)}\label{fig_traj}
\end{figure}

In an optimal control design based on HJB equations, the most computationally
intensive part, i.e., solving the HJB equation, is done off-line. Once the
equation is solved, the real-time closed-loop control can be computed using
interpolation with a minimum computational load. In the following simulations of
closed-loop control, the solution of the HJB equation in Example~II is
integrated with a MPC to stabilize the system at a desired optimal attitude. We
adopt a basic zero-order hold MPC controller. The sampling rate is 10 Hz. In
each time interval the optimal control $u^{*}$ is computed using interpolation
on the sparse grid. It is assumed that the measurement of $v$ and $\omega$ is
noise-corrupted. The noise has a uniform distribution with a magnitude about
$0.5\%$ of the maximum state value. Several trajectories under the closed-loop
control are shown in Figures~\ref{fig_traj1}--\ref{fig_traj3}.

\begin{figure}
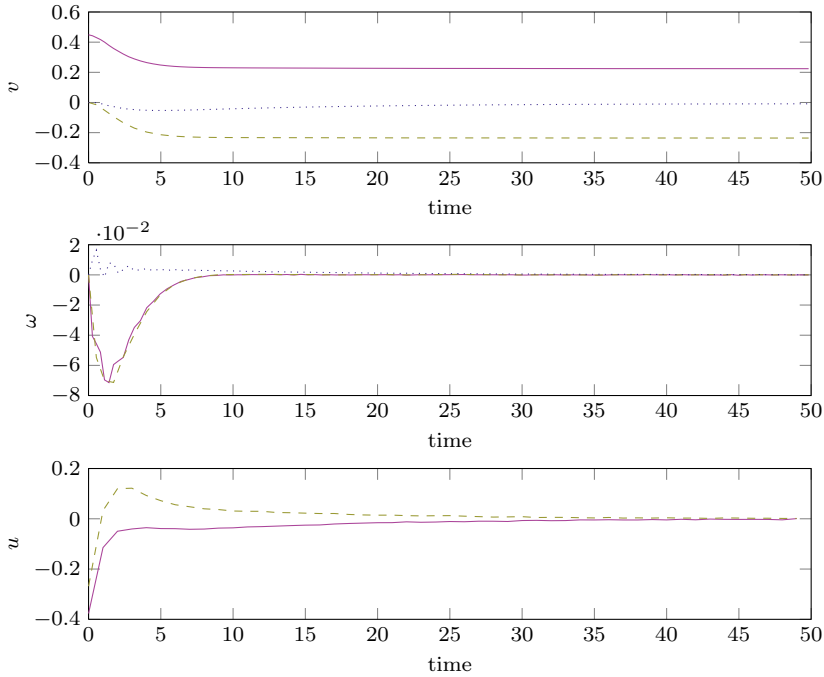

  \centering
%
   \caption{Closed-loop trajectories (Solid~\ref{plot_solid}: $\phi$, $\omega_1$, $u_1$;
    dashed~\ref{plot_dashed}: $\theta$, $\omega_2$, $u_2$;
    dotted~\ref{plot_dotted}: $\psi$, $\omega_3$)}\label{fig_traj1}
\end{figure}
\begin{figure}
  \centering
%
%
   \caption{Closed-loop trajectories (Solid~\ref{plot_solid}: $\phi$, $\omega_1$, $u_1$;
    dashed~\ref{plot_dashed}: $\theta$, $\omega_2$, $u_2$;
    dotted~\ref{plot_dotted}: $\psi$, $\omega_3$)}\label{fig_traj2}
\end{figure}
\begin{figure}
  \centering
%
%
   \caption{Closed-loop trajectories (Solid~\ref{plot_solid}: $\phi$, $\omega_1$, $u_1$;
    dashed~\ref{plot_dashed}: $\theta$, $\omega_2$, $u_2$;
    dotted~\ref{plot_dotted}: $\psi$, $\omega_3$)}\label{fig_traj3}
\end{figure}

\subsection{Example~III}
In previous examples, the closed-loop control is based on a fixed
horizon, $[0, T]$. The MPC algorithm uses the value of $u(0,x)$ in
feedback. In the following, we exemplify the method of computing
$u(t,x)$ for $t$ in a given interval $[0, t_1]$, where $t_1\leq
T$. Because of the causality free property, it is straightforward to
include the time variable as an additional dimension in the sparse
grid. Consider the following optimal control problem
\begin{align*}
\dot x_1 &= -x_1+x_2,\\
\dot x_2 &= -x_2 +\frac{x_3}{1+x_1^2+x_2^2}, \\
\dot x_3 &= \left(-2x_1^2+2x_1x_2-2x_2^2+\frac{2x_2x_3}{1+x_1^2+x_2^2}  \right) \frac{x_3}{1+x_1^2+x_2^2} + (1+x_1^2+x_2^2)u,
\end{align*}
where the cost functional is
\begin{equation*}
 \frac{1}{2} \int_{t_0}^T \frac{x_3^2}{{(1+x_1^2+x_2^2)}^2} +u^2\; dt.
\end{equation*}
For any initial condition, $x(t_0)=x_0$, an optimal control exists
that achieves a minimum value, $V(t_0,x_0)$. Let $z =(t,x)$ represents
a point in the time-state space. The sparse grid is generated in the
following region in $\Re\times \Re^3$
\[D=\left\{ z=(t,x) \in \Re\times \Re^3: \;\; t\in [0, 5] \mbox{ and } x_i\in [-2, 2], \mbox{ for } i=1,2,3 \right\}.\]
For each grid point $z^\bi_\bj$, the minimum value, $V$, and the
associated costate, $\vec{\lambda}$, at $z^\bi_\bj$ are computed by
solving \eqref{eq:hamchar}. Different from previous examples, the
sparse grid is generated in $(t, x)$-space so that one can approximate
$V(t, x)$ using an interpolation for $t\geq 0$ in a finite
interval. As a result, the MPC feedback can be applied using a
variable time horizon.

The problem has a known solution
\[ V(t,x)=\frac{1}{2} \frac{x_3^2}{(1+x_1^2+x_2^2)^2} \tanh (T-t), \;\;  u^\ast (t, x,)=-\frac{x_3}{(1+x_1^2+x_2^2)}\tanh (T-t). \]
It is used to check the accuracy of the numerical result.  In the
computation, we adopted a CGL sparse grid with $q=12$. The number of
grid points for each dimension is $2^{q-4}+1=257$. The total number of
grid points in the $4$-D domain is \[\abs{\Gsp^q}=18,945,\]
in comparison to dense grids, a finite difference discretization using
$33$ points in each dimension results in a grid size greater than
$1.18\times 10^6$, which is impractical for personal computers.

At all points in the sparse grid, we set the error tolerance to be
$10^{-7}$.  The upper bound of $\einterp+\ebvp$ is numerically
approximated using the method in Section~\ref{sec_totalerror}. More
specifically, $1200$ points are randomly generated in $D$. At each
sample point,  the value of $ V(t,x)$ is computed using
interpolation. The MAE is $8.5\times
10^{-4}$. Figure~\ref{fig_hist_exampleIII} is the histogram of the
approximation error, which shows a distribution concentrated around
the MAE with a small variance $2.8\times 10^{-6}$.
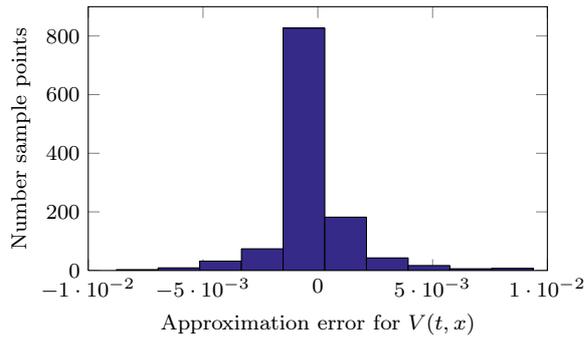
\begin{figure}
  \centering
\begin{tikzpicture}

\begin{axis}[%
  width=3.0in,
  height=2in,
  scaled ticks=false,
  colormap={mymap}{[1pt] rgb(0pt)=(0.2081,0.1663,0.5292); rgb(1pt)=(0.211624,0.189781,0.577676); rgb(2pt)=(0.212252,0.213771,0.626971); rgb(3pt)=(0.2081,0.2386,0.677086); rgb(4pt)=(0.195905,0.264457,0.7279); rgb(5pt)=(0.170729,0.291938,0.779248); rgb(6pt)=(0.125271,0.324243,0.830271); rgb(7pt)=(0.0591333,0.359833,0.868333); rgb(8pt)=(0.0116952,0.38751,0.881957); rgb(9pt)=(0.00595714,0.408614,0.882843); rgb(10pt)=(0.0165143,0.4266,0.878633); rgb(11pt)=(0.0328524,0.443043,0.871957); rgb(12pt)=(0.0498143,0.458571,0.864057); rgb(13pt)=(0.0629333,0.47369,0.855438); rgb(14pt)=(0.0722667,0.488667,0.8467); rgb(15pt)=(0.0779429,0.503986,0.838371); rgb(16pt)=(0.0793476,0.520024,0.831181); rgb(17pt)=(0.0749429,0.537543,0.826271); rgb(18pt)=(0.0640571,0.556986,0.823957); rgb(19pt)=(0.0487714,0.577224,0.822829); rgb(20pt)=(0.0343429,0.596581,0.819852); rgb(21pt)=(0.0265,0.6137,0.8135); rgb(22pt)=(0.0238905,0.628662,0.803762); rgb(23pt)=(0.0230905,0.641786,0.791267); rgb(24pt)=(0.0227714,0.653486,0.776757); rgb(25pt)=(0.0266619,0.664195,0.760719); rgb(26pt)=(0.0383714,0.674271,0.743552); rgb(27pt)=(0.0589714,0.683757,0.725386); rgb(28pt)=(0.0843,0.692833,0.706167); rgb(29pt)=(0.113295,0.7015,0.685857); rgb(30pt)=(0.145271,0.709757,0.664629); rgb(31pt)=(0.180133,0.717657,0.642433); rgb(32pt)=(0.217829,0.725043,0.619262); rgb(33pt)=(0.258643,0.731714,0.595429); rgb(34pt)=(0.302171,0.737605,0.571186); rgb(35pt)=(0.348167,0.742433,0.547267); rgb(36pt)=(0.395257,0.7459,0.524443); rgb(37pt)=(0.44201,0.748081,0.503314); rgb(38pt)=(0.487124,0.749062,0.483976); rgb(39pt)=(0.530029,0.749114,0.466114); rgb(40pt)=(0.570857,0.748519,0.44939); rgb(41pt)=(0.609852,0.747314,0.433686); rgb(42pt)=(0.6473,0.7456,0.4188); rgb(43pt)=(0.683419,0.743476,0.404433); rgb(44pt)=(0.71841,0.741133,0.390476); rgb(45pt)=(0.752486,0.7384,0.376814); rgb(46pt)=(0.785843,0.735567,0.363271); rgb(47pt)=(0.818505,0.732733,0.34979); rgb(48pt)=(0.850657,0.7299,0.336029); rgb(49pt)=(0.882433,0.727433,0.3217); rgb(50pt)=(0.913933,0.725786,0.306276); rgb(51pt)=(0.944957,0.726114,0.288643); rgb(52pt)=(0.973895,0.731395,0.266648); rgb(53pt)=(0.993771,0.745457,0.240348); rgb(54pt)=(0.999043,0.765314,0.216414); rgb(55pt)=(0.995533,0.786057,0.196652); rgb(56pt)=(0.988,0.8066,0.179367); rgb(57pt)=(0.978857,0.827143,0.163314); rgb(58pt)=(0.9697,0.848138,0.147452); rgb(59pt)=(0.962586,0.870514,0.1309); rgb(60pt)=(0.958871,0.8949,0.113243); rgb(61pt)=(0.959824,0.921833,0.0948381); rgb(62pt)=(0.9661,0.951443,0.0755333); rgb(63pt)=(0.9763,0.9831,0.0538)},
xmin=-0.01,
xmax=0.01,
xlabel={Approximation error for $V(t,x)$},
ymin=0,
ymax=900,
ylabel={Number sample points},
]

\addplot[table/row sep=crcr,patch,patch type=rectangle,shader=flat corner,draw=black,forget plot,patch table with point meta={%
1	2	3	4	1\\
6	7	8	9	1\\
11	12	13	14	1\\
16	17	18	19	1\\
21	22	23	24	1\\
26	27	28	29	1\\
31	32	33	34	1\\
36	37	38	39	1\\
41	42	43	44	1\\
46	47	48	49	1\\
}]
table[row sep=crcr] {%
x	y\\
-0.00877488885051875	0\\
-0.00877488885051875	0\\
-0.00877488885051875	3\\
-0.00695909822414966	3\\
-0.00695909822414966	0\\
-0.00695909822414966	0\\
-0.00695909822414966	0\\
-0.00695909822414966	9\\
-0.00514330759778058	9\\
-0.00514330759778058	0\\
-0.00514330759778058	0\\
-0.00514330759778058	0\\
-0.00514330759778058	32\\
-0.00332751697141149	32\\
-0.00332751697141149	0\\
-0.00332751697141149	0\\
-0.00332751697141149	0\\
-0.00332751697141149	74\\
-0.00151172634504241	74\\
-0.00151172634504241	0\\
-0.00151172634504241	0\\
-0.00151172634504241	0\\
-0.00151172634504241	827\\
0.000304064281326677	827\\
0.000304064281326677	0\\
0.000304064281326677	0\\
0.000304064281326677	0\\
0.000304064281326677	182\\
0.00211985490769576	182\\
0.00211985490769576	0\\
0.00211985490769576	0\\
0.00211985490769576	0\\
0.00211985490769576	43\\
0.00393564553406485	43\\
0.00393564553406485	0\\
0.00393564553406485	0\\
0.00393564553406485	0\\
0.00393564553406485	17\\
0.00575143616043393	17\\
0.00575143616043393	0\\
0.00575143616043393	0\\
0.00575143616043393	0\\
0.00575143616043393	6\\
0.00756722678680302	6\\
0.00756722678680302	0\\
0.00756722678680302	0\\
0.00756722678680302	0\\
0.00756722678680302	7\\
0.0093830174131721	7\\
0.0093830174131721	0\\
0.0093830174131721	0\\
};
\end{axis}
\end{tikzpicture}
  \caption{The histogram of approximation errors for $V(t,x)$ in
    $D$}\label{fig_hist_exampleIII}
\end{figure}

In the MPC feedback, we use a variable time
window. Figure~\ref{fig_closedloop_exampleIII} is a trajectory under
the closed-loop control. The initial time is $t_0=0$. Then the control
input is updated using a fixed time step, $\Delta t =1/7$. At each
$t_k < 5$, the time window for the feedback is $T-t_k$. An
interpolation on the sparse grid is used to compute
$\vec{\lambda}(t_k, x_k)$. Then control input $u(t_k,x_k)$ is applied
to the system. At $t=5$, the initial time is set to zero and the
process repeats. In the simulations, random noise of uniform
distribution in $[-0.05, 0.05]$ is added to the state variables.
\begin{figure}
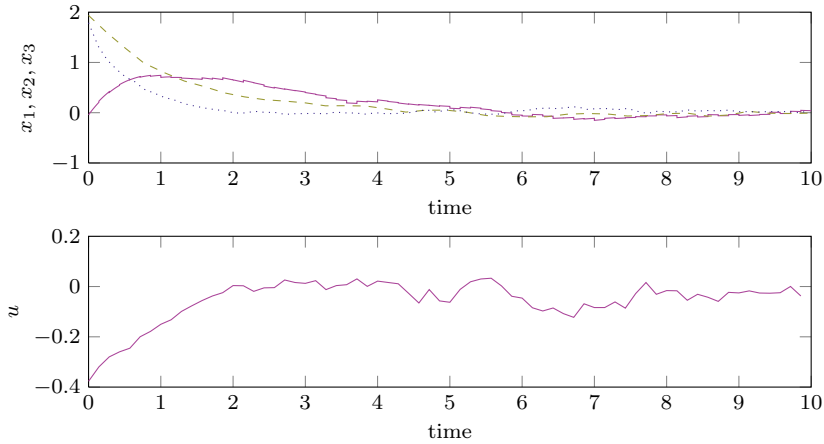

  \centering
  \caption{ Closed-loop trajectories (Solid~\ref{plot_solid}: $x_1$ and $u$;
    dashed~\ref{plot_dashed}: $x_2$;
    dotted~\ref{plot_dotted}: $x_3$)}\label{fig_closedloop_exampleIII}
\end{figure}

\section{Conclusions}
The characteristics method addressed in this paper is causality free and has
perfect parallelism. It can be easily integrated in a parallel computer with any
grid, such as sparse grids to mitigate the curse of dimensionality. The examples
show that the algorithm is tractable for $6$-D HJB equations. In the case of the
attitude control of rigid body with six state variables, a solution is computed
on a sparse grid with about $4.5 \times 10^5$ points, whereas the corresponding
dense grid has more than $10^{12}$ points.
In another example, the method is applied to a sparse grid in the
time-state space so that the solution can be approximated within a
given time interval.
A theorem on the error upper bound is
proved. In addition to the parallelism, another advantage of the causality free
method is that the accuracy of the interpolated solution can be numerically
checked pointwise. The effectiveness of the numerical solution for closed-loop
control is demonstrated in Example~II using a MPC controller.

For future research, an interesting question is how to generalize the
idea to problems with an infinite time horizon. Another interesting
research topic is to solve minimum-time problems. In the presence of
multiple solutions, identifying the globally or locally optimal
solution remains a challenge that needs to be addressed.

\begin{acknowledgements}
  This work was supported in part by AFOSR, DARPA, NRL, and CRUSER of Naval
  Postgraduate School.  Thanks to Carlos~F. Borges for his
  time discussing terminology and the naming of the method. Thanks to Arthur~J.
  Krener for his insights on HJB solutions. Thanks to Lars Gr\"une for his
  comments on the parallel computation of reachable set~\cite{jahn}.
\end{acknowledgements}

\bibliographystyle{spmpscinat}

\end{document}